\documentclass[11pt]{amsart}
\usepackage{amssymb,amsmath,amscd,amsfonts,amsthm,mathrsfs}
\usepackage{verbatim} 
\usepackage{tikz}
\usepackage{enumitem}
\usetikzlibrary{arrows,positioning,automata,shadows}
\input xy
\xyoption{all}
\usepackage[all]{xy}  
\usepackage{fullpage}
\usepackage{cases}
\newcommand*{\dosubeqns}{}
  \makeatletter
  \@ifdefinable \subeqnstoks {\newtoks \subeqnstoks}
  
  \DeclareRobustCommand*{\settheparentequation}{%
    \def \theparentequation
  }
  \newenvironment{subeqns*}[1][gather]{%
    \subequations
    \the \subeqnstoks
    \renewcommand*{\dosubeqns}[1]{\begin{#1}##1\end{#1}}%
    \collect@body \dosubeqns
  }{%
    \protected@edef \@tempa {%
      \global \subeqnstoks {%
        \protect \setcounter{parentequation}{\number\value{parentequation}}%
        \settheparentequation {\theparentequation}%
        \protect \setcounter{equation}{\number\value{equation}}%
      }%
    }%
    \@tempa
    \endsubequations
  }
  \makeatother

\usepackage{dsfont}
\usepackage{color}
\usepackage{pgf}
\usepackage{tikz}
\usepackage[T1]{fontenc}
\usepackage{lmodern}
\usetikzlibrary{arrows,automata}
\usepackage[francais]{babel}

\newtheorem{theorem}{Th\'eor\`eme}[section]
\newtheorem{proposition}[theorem]{Proposition}
\newtheorem{lemma}[theorem]{Lemme}
\newtheorem{remark}[theorem]{Remarque}
\newtheorem{example}[theorem]{Exemple}
\newtheorem{definition}[theorem]{D\'efinition}
\newtheorem{corollary}[theorem]{Corollaire}

\newcommand{\N}{\mathbb{N}}

\newcommand{\Z}{\mathbb{Z}}

\newcommand{\R}{\mathbb{R}}

\newcommand{\Dc}{\mathcal{D}}

\def \ds{\displaystyle}

\def \rmi{{\rm i}}
\def \M{\Sigma}
\def \A{\mathcal{A}}

\def \Dc{\text{D}_c}
\def \Dc1{\text{D}_d}
\def \Dc2{\mathcal{D}}

\title{Sur les occurrences des mots dans les nombres premiers}
\author{Gautier Hanna}
\address{1. Universit\'e de Lorraine, Institut Elie Cartan de Lorraine, UMR 7502, Vandoeuvre-l\`es-Nancy, F-54506, France;
2. CNRS, Institut Elie Cartan de Lorraine, UMR 7502, Vandoeuvre-l\`es-Nancy, F-54506, France}
\email{gautier.hanna@univ-lorraine.fr}



\begin{document}

\maketitle

\begin{abstract}
In this paper, we generalize Mauduit and Rivat's theorem on the Rudin-Shapiro sequence. Weakening the hypothesis needed in their theorem, we prove a prime number theorem for a large class of functions defined on the digits. Our result covers the case of generalized Rudin-Shapiro sequences as well as bloc-additive sequences on finite and infinite expansions. We also give a partial answer to a question posed by Kalai. 
\end{abstract}

\section{Introduction}

Mauduit et Rivat ont démontré dans \cite{MR:gelfond} une conjecture vieille de 40 ans due à A.Gelfond \cite{gelfond}. Elle stipulait en particulier que la moitié des nombres premiers avaient un nombre de $1$ pair dans l'écriture en base $2$. Cette question est reliée à l'étude des fonctions définies sur les chiffres en base $q$ (ici la fonction $s_2(n)$, la somme des chiffres en base $2$) et des sous-suites de suites automatiques \cite{AlloucheShallit} (la suite de Thue-Morse). La recherche d'un théorème des nombres premiers pour des fonctions définies sur les chiffres, voire l'étude des sous-suites de suites automatiques, est un problème ardu. La méthode développée dans \cite{MR:gelfond} a permis ces dernières années des progrès significatifs dans ce domaine \cite{MRD:gelfond3, ThueMorsenormal, digitales, RudinShapiro}.
 
Parall\`element, Kalai~\cite{Kalai1} s'est int\'eress\'e au probl\`eme suivant. \'Etant donn\'e $S$ un sous-ensemble de $\{1,\ldots,n\} $, $\mu(n)$ la fonction de M\"obius, et $\Omega_n$ l'espace de tous les $x=(x_1,\ldots,x_n)$ de $\{0,1\}^n$, a-t-on pour tout $A >0$,  \begin{equation*}\label{Kalai}
\hat{\mu} (S) := \ds\frac{1}{2^n}\ds\sum_{x \in \Omega_n}\mu (x_1 + 2x_2+\ldots +2^{n-1}x_n)(-1)^{\sum_{i \in S}x_i} = O(n^{-A})\; ?
\end{equation*} Bourgain~\cite{bourgain} a répondu positivement,  montrant notamment un principe d'al\'ea de M\"obius pour toutes les fonctions lin\'eaires sur $\Z / 2\Z$. Suite \`a ce travail, Kalai a demand\'e dans~\cite{Kalai2} d'\'etudier le cas des polyn\^omes de plus haut degr\'e, notamment le cas de la suite de Rudin-Shapiro~\cite{Rudin, Shapiro}. \'Etudier la suite de Rudin-Shapiro est naturel puisqu'il s'agit du cas le plus simple de polyn\^ome de degr\'e plus grand que $1$.  Si \begin{equation}\label{decompn}
	n = \ds\sum_{i \geq 0}\epsilon_i(n)q^i
\end{equation} est l'écriture de $n$ en base $q$, en utilisant~(\ref{decompn}) avec $q=2$, on pose $$a(n) = \ds\sum_{i \geq 0}\epsilon_i(n) \epsilon_{i+1}(n) $$ 
alors $(a(n) \bmod 2)_{n\geq 0}$ d\'esigne la suite de Rudin--Shapiro. Tao~\cite{Kalai2} a donn\'e une preuve d'un principe d'al\'ea de M\"obius dans ce cas particulier, et Mauduit et Rivat~\cite{RudinShapiro} ont donn\'e une formule asymptotique avec un terme d'erreur explicite et ont \'egalement obtenu un th\'eor\`eme des nombres premiers dans ce cas. Ils ont formul\'e deux conditions suffisantes sur une suite $(f(n))_{n \in \N}$ de module $1$ pour estimer de manière non triviale la somme $\sum_{n < N}\Lambda(n)f(n) $, o\`u $\Lambda$ d\'esigne la fonction de von Mangoldt. Pour notre part, nous allons altérer légèrement une de ces conditions.

Notons $\mathbb{U}$ le cercle unit\'e, $f^{(\lambda)}$ une troncation de la fonction $f$ (nous donnerons la d\'efinition pr\'ecise dans la partie~\ref{Section tous les resultats}) et  $e(x)= \exp(2\rmi \pi x)$.

\begin{definition}[Faible propri\'et\'e de petite propagation]\label{MR def 1}
On dit qu'une application $f : \mathbb{N} \rightarrow \mathbb{U}$ a la faible propri\'et\'e de petite propagation si, uniform\'ement pour $(\lambda, \kappa, \rho) \in \mathbb{N}^3$ avec $\rho < \lambda$, le nombre d'entiers $l$ satisfaisant  $0 \leq l < q^{\lambda}$ tels qu'il existe $(k_1,k_2) \in \{0,\ldots ,q^{\kappa}-1\}^2$ avec  \begin{equation}\label{petite propagation}f(lq^{\kappa}+k_1+k_2)\overline{f(lq^{\kappa}+k_1)} \neq f^{(\kappa + \rho)}(lq^{\kappa}+k_1+k_2)\overline{f^{(\kappa + \rho)}(lq^{\kappa}+k_1)}
\end{equation}  est $O\left(q^{\lambda-\rho + \log \rho}\right)$, la constante ne d\'ependant que de $q$ et $f$.
\end{definition}

\begin{definition}[Propri\'et\'e de Fourier]\label{propriete fourier mauduit rivat}
On dit qu'une application $f : \mathbb{N} \rightarrow \mathbb{U}$ a la \textit{propri\'et\'e de Fourier} s'il existe une fonction $\gamma$ croissante, avec $ \lim_{\lambda \rightarrow + \infty} \gamma(\lambda) =  + \infty$ et une constante absolue $c >0$ tels que pour tous entiers positifs $\lambda , \kappa $, avec $\kappa \leq c \lambda $ et tout r\'eel $t$, on ait :
\begin{equation*}\label{equation Mauduit Rivat}
\left|\frac{1}{q^\lambda}\ds\sum_{0 \leq n < q^\lambda}f(q^{\kappa}n)e(-nt)\right| \leq q^{- \gamma(\lambda)}.
\end{equation*}
\end{definition}

Nous avons alors le 

\begin{theorem}\label{notre gros theoreme}
Soit $f$ une application  v\'erifiant les d\'efinitions \ref{MR def 1} et \ref{propriete fourier mauduit rivat}. Alors $f$ vérifie uniform\'ement en $\vartheta \in \R $:
\begin{equation}\label{nouvelle equation von mangoldt}
\left| \ds\sum_{n \leq x}\Lambda (n)f(n)e(\vartheta n)\right| \ll c_1(q) (\log x)^{c_2(q)}xq^{-\gamma(2\lfloor (\log x)/80 \log q \rfloor)/20+ \log (\gamma(2\lfloor (\log x)/80 \log q \rfloor)/20)},
\end{equation} 
avec 
\begin{align*}
  c_1(q) &= \max (\tau(q)\log q,\log^{10}q)^{1/4}(\log q)^{2-2c_2(q)},\\
	c_2(q) &= 4+\frac{\log q}{4}+ \frac{1}{4}\max (\omega (q), 2).
\end{align*}
\end{theorem} 
Ces énoncés diffèrent de \cite{RudinShapiro} par une altération dans définition \ref{MR def 1} de $O(q^{\lambda - \rho})$ en $O(q^{\lambda - \rho + \log \rho})$ et par l'altération des constantes $c_1(q)$ et $c_2(q)$. 

\begin{remark}\label{Remarque gros thm}
La démonstration du Théorème \ref{notre gros theoreme} repose sur l'estimation de sommes de type I et de type II et en l'application du Lemme $1$ de~\cite{MR:gelfond}. En particulier les techniques utilisées permettent d'avoir  l'estimation \eqref{nouvelle equation von mangoldt} avec $\mu$ la fonction de Möbius à la place de $\Lambda$ en utilisant l'équation $13.40$ de \cite{Iwaniec}.
\end{remark}

Dans \cite{RudinShapiro}, Mauduit et Rivat utilisent leur théorème avec $f(n) = e(\alpha a(n)) $ et $(a(n))_{n \in \N}$ une suite de Rudin--Shapiro g\'en\'eralis\'ee, ils traitent les deux cas suivants : 
\begin{equation}\label{gen1}
  \beta_\delta (n) = \ds\sum_{l \geq 0}\epsilon_l(n)\epsilon_{l +\delta}(n)
\end{equation}
et
\begin{equation}\label{gen2}
  b_d(n) = \ds\sum_{l \geq 0}\epsilon_l(n)\epsilon_{l+1}(n)\cdots \epsilon_{l+d}(n).
\end{equation}	
Notons que la suite $\beta_\delta (n)$ a \'et\'e introduite par Allouche et Liardet dans~\cite{AlloucheLiardet}. Le Théorème~\ref{notre gros theoreme} implique que ces suites sont uniformément distribuées et possèdent un théorème des nombres premiers (pour les énoncés exacts, voir les corollaires \ref{coro1}, \ref{coro2} et \ref{coro3}).

\bigskip
Dans cet article nous g\'en\'eralisons les r\'esultats de~\cite{RudinShapiro} \`a $$a(n) = a(\lfloor n/q^{T_q(n)-\beta} \rfloor)+ \ds\sum_{0 \leq l \leq T_q(n)-\beta}h(\epsilon_l(n),\epsilon_{l+1}(n),\ldots, \epsilon_{l+\beta-1}(n)),$$
o\`u $h$ est une fonction à $\beta$ variables o\`u $\beta$ est un entier $\geq 2$ et $T_q(n) = \lfloor \log n / \log q \rfloor$ est la taille de $n$ en base $q$. La forme de ces suites, que nous nommons $\beta$-r\'ecursives, g\'en\'eralise~(\ref{gen1}) et~(\ref{gen2}), mais \'egalement le cas des suites digitales,  parfois nomm\'ees bloc additives, qu'on peut trouver dans~\cite{Cateland}. La recherche d'un principe d'al\'ea de M\"obius pour les suites bloc additives a \'et\'e traitée par M\"ullner~\cite{Muellner}. Du fait de leur forme, les suites $\beta$-r\'ecursives permettent de mieux r\'epondre que les fonctions digitales \`a la question de Kalai (Th\'eor\`eme \ref{Thm Kalai}). 
 
\bigskip

Le lecteur trouvera dans la partie~\ref{section definition} une introduction aux suites $\beta$-r\'ecursives et aux diff\'erentes notations qui seront utilis\'ees dans l'article. Dans la partie \ref{Section tous les resultats} nous développons plus précisément les conséquences du théorème principal (Th\'eor\`eme \ref{notre gros theoreme}). La condition de faible propagation obtenue, et l'explication de l'alt\'eration de la condition initiale sont situ\'ees dans la partie~\ref{section petite propa} de ce travail. Comme nous altérons les définitions de \cite{RudinShapiro}, nous sommes obligés de reprendre partiellement cet article. C'est ce qui est fait dans la partie~\ref{Section TP}. Pour terminer l'\'etude des suites $\beta$-r\'ecursives, la condition de Fourier est v\'erifi\'ee dans la partie~\ref{section genealogie}. Pour se faire, nous sommes amenés \`a contr\^oler la norme infinie d'une matrice reliée à la suite $\beta$ récursive. Nous donnons la formule exacte de la norme infinie de cette matrice (Proposition \ref{prop equa a prouver thm principal}) en exhibant un graphe. La partie~\ref{Section : applis} est d\'evou\'ee \`a la collecte de r\'esultats.

\section{Notations et d\'efinitions}\label{section definition}

Soit $q $ un entier sup\'erieur ou \'egal \`a $2$. On note $\mathcal{A}$ l'alphabet $\mathcal{A} := \{0,\ldots , q-1\}$. On note $\M$ l'ensemble des mots sur $\A$, $\M^*$ l'ensemble des mots finis, $\M_k$ l'ensemble des mots  de taille $k$, $\M_k^*$ l'ensemble des mots de taille au plus $k$, et $\epsilon$ le mot de taille $0$. Ainsi $\M_0 = \{ \epsilon \}$, $\M_1 = \{0,\ldots , q-1\}$, $\M_1^* = \{ \epsilon , 0, \ldots , q-1 \}$, etc. Soient $\omega , \omega' \in \M $. On note $ \omega\cdot\omega'$ leur concat\'enation et $| \omega |$ la taille de $\omega$ (on omettra parfois le symbole $\cdot$, toutefois sans risque de confusion). Pour un entier $k \geq 0 $, on note $\overline{\omega}^k$, le pr\'efixe de $\omega$ de taille $k$, et $\underline{\omega}_k$ son suffixe de taille $k$. On a par convention $\overline{\omega}^0 = \underline{\omega}_0 = \epsilon$. Ainsi, pour tout entier $k$ entre $0$ et $|\omega |$, on a la décomposition $\omega = \overline{\omega}^{|\omega|-k}\cdot\underline{\omega}_k$. On note $\epsilon_i(\omega)$ la $i$-i\`eme lettre de $\omega$, lu de droite \`a gauche, donc $\omega = \epsilon_{|\omega|-1}(\omega)\cdots\epsilon_0(\omega) $.

On d\'efinit l'application $\varphi : \M^* \rightarrow \mathbb{N}$ par $\varphi(\omega) = \ds\sum_{i = 0}^{|\omega|-1}\epsilon_i(\omega)q^i$. Pour $r \in \A $, on utilisera la notation $\hat{r}$ pour l'entier $\varphi (r) $. Pour un entier $x$ compris entre $0$ et $q-1$, on note $\dot{x} = \varphi^{-1}(x) $ pour d\'esigner la lettre correspondante. Par exemple, pour $\omega = 280163$, on a $|\omega| = 6, \, \overline{\omega}^2 = 28, \, \underline{\omega}_3 = 163$, et pour la base $q=11$,
$$\varphi(\omega)= 2*11^5 +8*11^4 + 0*11^3 + 1*11^2 + 6*11^1 + 3*11^0 = 439420.$$
Enfin, soit $\omega' \in \M^*$, nous notons $\mathds{1}_{\omega'}: \M^* \rightarrow \{0,1\}$ avec
$$\mathds{1}_{\omega'}(\omega)
 = \begin{cases} 1 &\mbox{si } \omega= \omega'; \\
0 & \mbox{sinon. }\end{cases}
$$

Nous introduisons maintenant l'objet central de l'\'etude de cet article.

\begin{definition}\label{def beta rec}
Soient $(a(n))_{n \in \N}$ une suite \`a valeurs dans $\Z$ et  $\beta$ un entier sup\'erieur ou \'egal \`a  $2$. On dit que $(a(n))_{n \in \N}$ est $\beta$-r\'ecursive s'il existe  une application $g : \M_\beta \rightarrow \mathbb{N}$ telle que pour tout $n \geq 1$ et pour tout $\omega$ dans $\M_\beta$, on ait : \begin{equation}\label{condi 1 beta rec}
a(q^{\beta}n + \varphi(\omega)) = a(q^{\beta - 1}n + \varphi(\overline{\omega}^{|\omega|-1}))+g(\omega),
\end{equation} et telle que si  $\overline{\omega}^1 \neq 0$ : \begin{equation}
\label{condi 2 beta rec}
a(\varphi(\omega)) = a(\varphi(\overline{\omega}^{|\omega|-1}))+g(\omega).
\end{equation} 

Nous dirons que $g$ est la fonction de propagation de $a$.
\end{definition}

Comme $\omega$ est un élément de $\M_{\beta}$, si $\overline{\omega}^1 = 0$  il existe $\tilde{\omega}$ dans $\M_{\beta-1}^*$ tel que  $$\varphi(\omega) = \varphi(\tilde{\omega}) = \ds\sum_{i = 0}^{\beta - 2}\epsilon_i(\omega)q^i < q^{\beta - 1}.$$  Ainsi la notion de $\beta$-r\'ecursivit\'e n'impose de contrainte que pour les entiers au moins \'egaux \`a $q^{\beta-1} $. 

Citons ici trois exemples de suites $\beta$-récursives :

\begin{enumerate}[label=\Alph*]
\item[(E1)]\label{Exemple 1} \textbf{Suites de Rudin--Shapiro g\'en\'eralis\'ees.} Les suites de type Rudin-Shapiro, constituées des généralisations proposées par M. Queffélec~\cite{Queffelec}, par Grant, Shallit et Stoll~\cite{GSS}, ou encore Allouche et Liardet~\cite{AlloucheLiardet} sont des suites $\beta$-récursives. 
\item[(E2)] \textbf{Suites bloc additives.} Les suites digitales, parfois nommées blocs additives~\cite{AlloucheShallit, Cateland} définies par : $a(n) = \ds\sum_{i \geq 0}g(\epsilon_{i+\beta-1}(n)\cdot\ldots \cdot\epsilon_i(n)) $ avec $g(0\cdot \ldots \cdot 0) = 0$ et \eqref{decompn} sont des suites $\beta$-récursives.
\item[(E3)] \textbf{Suites bloc additives finies.} On peut se passer de la condition $g(0\cdot\ldots \cdot 0) = 0$ et prendre la suite $a(n) = \ds\sum_{i = 0}^{T_q(n)-r}g(\epsilon_{i+\beta-1}(n),\ldots,\epsilon_i(n))$, ce qui est fondamental si on veut compter les blocs de chiffres en écriture finie (par exemple la suite bloc-additive qui compte le nombre de $01$ vaudra $2$ pour $101$, ce qui est contraire à l'intuition). Cette suite est également une suite $\beta$-récursive. 
\end{enumerate}

L'objet de la proposition suivante est de faire le lien entre la définition \ref{def beta rec} et les trois exemples précités :

\begin{proposition}\label{Decomposition pour une suite beta recursive}
Soient un entier $\beta \geq 2$ et $(a(n))_{n \in \N}$ une suite $\beta$-r\'ecursive. Soit $n$ un entier, nous consid\'erons sa d\'ecomposition en base $q$, $$n = \ds\sum_{i = 0}^N \epsilon_i(n)q^i = \varphi\Big(\epsilon_N(n)\cdot \ldots \cdot\epsilon_1(n)\cdot \epsilon_0(n)\Big),$$ o\`u $N = T_q(n)$. Alors, si $n \geq q^{\beta - 1} $, on a $N \geq \beta - 1 $ et $$a(n) = a\bigg(\varphi\Big(\epsilon_N(n)\cdot\epsilon_{N-1}(n)\cdot\ldots\cdot\epsilon_{N-\beta+2}(n)\Big) \bigg)+\ds\sum_{l=0}^{N-\beta+1}g\Big(\epsilon_{l+\beta-1}(n)\cdot\ldots\cdot\epsilon_{l+1}(n)\cdot\epsilon_l(n)\Big). $$ 
Remarquons que $|\epsilon_N(n)\cdot\epsilon_{N-1}(n)\cdot\ldots\cdot\epsilon_{N-\beta+2}(n)| = \beta - 1$ et que $\epsilon_N(n) \neq 0$.
\end{proposition}

\begin{proof}
Le cas $N=\beta-1$ est immédiat, nous pouvons désormais supposer $N > \beta - 1$. Nous allons montrer par r\'ecurrence sur $r$ que pour tout entier $r $ compris entre $0$ et $N-\beta $, 

\begin{align}\label{HR1}
a\left(\ds\sum_{i = 0}^N \epsilon_i(n)q^i\right) &= a\bigg(q^{\beta - 1}\ds\sum_{i = \beta + r}^{N}\epsilon_i(n)q^{i - \beta - r}+ \varphi\Big(\epsilon_{\beta - 1 + r}(n)\cdot \ldots \cdot \epsilon_{r+1}(n) \Big) \bigg) \nonumber\\
&\qquad\qquad+ \ds\sum_{l = 0}^rg(\epsilon_{l+\beta-1}(n)\cdot\ldots\cdot\epsilon_l(n)). 
\end{align}

Pour $r = 0$, on a par (\ref{condi 1 beta rec}) : 

\begin{align*}
a\left(\ds\sum_{i = 0}^N \epsilon_i(n)q^i\right) &= a\bigg(q^{\beta}\ds\sum_{i = \beta}^{N}\epsilon_i(n)q^{i - \beta}+ \varphi\Big(\epsilon_{\beta - 1 }(n)\cdot \ldots \cdot \epsilon_{1}(n)\cdot\epsilon_0(n) \Big) \bigg)
\\ &= a\bigg(q^{\beta - 1}\ds\sum_{i = \beta}^{N}\epsilon_i(n)q^{i - \beta}+ \varphi\Big(\epsilon_{\beta - 1 }(n)\cdot \ldots \cdot \epsilon_{1}(n) \Big) \bigg) + g(\epsilon_{\beta-1}(n)\cdot\ldots\cdot\epsilon_0(n)).
\end{align*}

Supposons l'hypoth\`ese de r\'ecurrence (\ref{HR1}) satisfaite pour un certain $r \leq N-\beta - 1 $ et montrons (\ref{HR1}) pour $r+1$. Comme $\beta + r + 1 \leq N $, $$\ds\sum_{i = \beta + r + 1}^N \epsilon_i(n)q^{i-\beta-r-1} \geq 1 .$$ Alors, en utilisant  (\ref{HR1}) puis (\ref{condi 1 beta rec}) pour $r>0$ on obtient : 
\begin{align*}
 a\left(\ds\sum_{i = 0}^N \epsilon_i(n)q^i\right) &= a\bigg(q^{\beta - 1}\ds\sum_{i = \beta + r}^{N}\epsilon_i(n)q^{i - \beta - r}+ \varphi\Big(\epsilon_{\beta - 1 + r}(n)\cdot \ldots \cdot \epsilon_{r+1}(n) \Big) \bigg)\\
&\qquad\qquad + \ds\sum_{l = 0}^rg(\epsilon_{l+\beta-1}(n)\cdot\ldots\cdot\epsilon_l(n))
\\ &= a\bigg(q^{\beta - 1 }\ds\sum_{i = \beta + r+1}^{N}\epsilon_i(n)q^{i - \beta - r-1}+ \varphi\Big(\epsilon_{\beta+r}(n)\cdot\epsilon_{\beta - 1 + r}(n)\cdot \ldots \cdot \epsilon_{r+2}(n) \Big) \bigg) \\
&\qquad\qquad + \ds\sum_{l = 0}^{r+1}g(\epsilon_{l+\beta-1}(n)\cdot\ldots\cdot\epsilon_l(n)),
\end{align*}
ce qui conclut la récurrence.
En appliquant  (\ref{HR1}) \`a $r = N - \beta $, on obtient : 
\begin{align*}
a(n)  &= a\bigg(q^{\beta - 1}\epsilon_N(n)+ \varphi\Big(\epsilon_{N - 1 }(n)\cdot \ldots \cdot \epsilon_{N - \beta+1}(n) \Big) \bigg) + \ds\sum_{l = 0}^{N - \beta}g(\epsilon_{l+\beta-1}(n)\cdot\ldots\cdot\epsilon_l(n))
\\ &= a\bigg( \varphi\Big(\epsilon_N(n)\cdot \epsilon_{N - 1 }(n)\cdot \ldots \cdot \epsilon_{N - \beta+1}(n) \Big) \bigg) + \ds\sum_{l = 0}^{N - \beta}g(\epsilon_{l+\beta-1}(n)\cdot\ldots\cdot\epsilon_l(n)),
\end{align*}
et on conclut par (\ref{condi 2 beta rec}), parce que $N$ d\'esigne l'indice du dernier terme non nul dans la d\'ecomposition de $n$, ce qui veut dire $\epsilon_N(n) \neq 0$, et enfin $\epsilon_N(n)\cdot \epsilon_{N-1}(n)\cdots \epsilon_{N-\beta+1}(n)\in \Sigma_\beta$.
\end{proof}

\begin{remark}
Avec l'écriture de la proposition précédente, il suffit pour avoir les exemples de (E1) de prendre la fonction $g$ correspondante. Par exemple pour retrouver $a(n) = \ds\sum_{i \geq 0}\epsilon_{i+\delta}(n)\epsilon_i(n) = \beta_\delta(n)$ (suite d'Allouche et Liardet), on pose $$g(\omega) = \ds\sum_{(i_1,\ldots,i_{\delta-1})\in \{0,\ldots,q-1\}^{\delta-1}}\ds\widehat{\epsilon_{\delta}(\omega)} \mathds{1}_{\widehat{\epsilon_{\delta-1}(\omega)}=i_{\delta-1}} \ldots \mathds{1}_{\widehat{\epsilon_{1}(\omega)}=i_{1}} \widehat{ \epsilon_0(\omega)}.$$ Quant à l'exemple (E2) il suffit de prendre \begin{equation}\label{forme bloc additive}a(\varphi(\omega)) = \ds\sum_{i = 1}^{\beta - 1}g(0^{\beta - i}\cdot \overline{\omega}^{ i})\end{equation} si  $\overline{\omega}^1 \neq 0$ et $a(\varphi(\omega)) = 0$ si $\overline{\omega}^1 = 0$.
\end{remark}

\section{R\'esultats principaux}\label{Section tous les resultats}

Pour cette partie nous rappelons que $\tau(n)$ d\'esigne le nombre de diviseurs de $n$, et $\omega(n)$ le nombre de facteurs premiers dans la d\'ecomposition de $n$ (ainsi $\omega(2^2*3) = 2$). Les deux notations $\omega$ pour désigner un mot et la suite arithmétique $\omega (n)$ ne se recoupent pas dans l'article, et nous pouvons utiliser conjointement ces deux notations sans risque de confusion. De plus nous notons $\pi(x;a,m):= \#\{p \leq x : p \equiv a \bmod m\}$.

Il nous est n\'ecessaire par la suite de d\'efinir des fonctions tronqu\'ees et de travailler sur le cercle unité $\mathbb{U}$. En effet, notre article repose sur les r\'esultats de~\cite{RudinShapiro} qui utilisent des sommes d'exponentielles, et o\`u le principe de troncation est essentiel. 

\begin{definition}
Soit $(a(n))_{n \in \N}$ une suite $\beta$-r\'ecursive, et soit $\lambda$ un entier naturel. On d\'efinit $(a^{(\lambda)}(n))_{n \in \N}$, la suite tronqu\'ee en $\lambda$, par $$a^{(\lambda)}(n) = a(n \, \text{mod} \, q^{\lambda}),$$ o\`u $n \, \text{mod} \, q^{\lambda} $ d\'esigne le reste de la division euclidienne de $n$ par $q^{\lambda} $. Soit $\alpha$ un nombre r\'eel. On d\'efinit les applications $f : \N \rightarrow \mathbb{U} $ et $f^{(\lambda)} : \N \rightarrow \mathbb{U} $ par  $$f(n)=e\big(\alpha a(n)\big) \quad \text{et}\quad f^{(\lambda)}(n)=e\big(\alpha a^{(\lambda)}(n)\big).$$ On dit qu'elles sont associ\'ees aux suites $(a(n))_{n \in \N}$ et $(a^{(\lambda)}(n))_{n \in \N}$.
\end{definition}

Les définitions \ref{MR def 1} et \ref{propriete fourier mauduit rivat} ainsi que le Théorème \ref{notre gros theoreme} prennent ici un sens rigoureux. Tout comme Mauduit et Rivat, nous d\'eduisons de ce th\'eor\`eme trois corollaires, dont les preuves sont identiques \`a~\cite[Corollary~1--3]{RudinShapiro} :

\begin{corollary}\label{coro1}
Soit $b : \N \rightarrow \N$ telle que pour tout $\alpha$ irrationnel, la fonction $f(n) = e(\alpha b(n))$ v\'erifie les d\'efinitions \ref{MR def 1} et \ref{propriete fourier mauduit rivat}. Alors pour tout entier relatif $a$ et tout entier naturel $m$ premier avec $a$, la suite $(\alpha b(p))_{p \in \mathcal{P}(a,m)}$ est uniform\'ement distribu\'ee si et seulement si $\alpha$ est irrationnel.
\end{corollary}

\begin{corollary}\label{coro2}
Soit $b : \N \rightarrow \N$ et $m$ et $m'$ des entiers plus grands que $1$, tels que pour tout $1 \leq j' < m'$, la fonction $f(n) = e\left(\frac{j'}{m'}b(n) \right)$ v\'erifie les d\'efinitions \ref{MR def 1} et \ref{propriete fourier mauduit rivat}. Alors, pour tous $a$ et $a'$ tel que $a$ soit premier avec $m$, on a, lorsque $x$ tend vers l'infini : $$\# \{p\leq x, \; p \in \mathcal{P}(a,m), \; b(p) \equiv a' \bmod m' \} = (1+o(1))\frac{\pi (x;a,m)}{m'} .$$

\end{corollary}

\begin{corollary}\label{coro3}
Soit $b : \N \rightarrow \N$ et $m$ et $m'$ des entiers plus grands que $1$, tels que pour tout $1 \leq j' < m'$, la fonction $f(n) = e\left(\frac{j'}{m'}b(n) \right)$ v\'erifie les d\'efinitions \ref{MR def 1} et \ref{propriete fourier mauduit rivat}. Alors, pour tout $a$ et $a'$ tel que $a$ soit premier avec $m$, la suite ($\vartheta p)_{\{ p \in \mathcal{P}(a,m),\; b(p) \equiv a' \bmod m' \}}$ est uniform\'ement distribu\'ee si et seulement si $\vartheta$ est irrationnel.
\end{corollary}

La particularit\'e des suites $\beta$-r\'ecursives permet de plus d'avoir le r\'esultat suivant, qui r\'epond partiellement \`a la question de Kalai (nous g\'en\'eralisons directement en base $q$ arbitraire) : 

\begin{theorem}\label{Thm Kalai}
Soit $k\geq 1$ et $$a(n) = \ds\sum_{i = 0}^{T_q(n)-k}P(\epsilon_{i+k}(n),\ldots,\epsilon_i(n)),$$ o\`u $\epsilon_0(n), \ldots, \epsilon_{T_q(n)}(n)$ sont les chiffres de $n$ en base $q$, et $P\in \mathbb{Z}[X_k,\ldots, X_0]$ est un polyn\^ome de degr\'e $d \leq k+1$ de la forme 
$$P(X_k,\ldots, X_0)=X_k X_0 P_1(X_k,\ldots, X_0)+P_2(X_k,\ldots, X_0),$$
o\`u $P_1, P_2\in \mathbb{Z}[X_k,\ldots, X_0]$ sont tels que l'\'equation $P_1(1,X_{k-1},\ldots,X_1,1)=1$ poss\`ede une solution $(X_{k-1},\ldots,X_1)\in\{0,1,\ldots,q-1\}^{k-1}$, et il n'y a pas de mon\^ome non nul divisible par $X_kX_0$ dans $P_2$. Alors 
\begin{equation}
\lim_{N \rightarrow \infty}\ds\frac{1}{N}\ds\sum_{n \leq N}\mu(n)(-1)^{a(n)} = 0.
\end{equation}
\end{theorem}

On remarque que $P(X_k,\ldots,X_0) = \ds\prod_{i=0}^k X_i$ ainsi que $P(X_k,\ldots,X_0) = \ds\prod_{i=0}^k(1- X_i)$ vérifient les conditions du théorème, \textit{a contrario} de $P(X_k,\ldots,X_0) = X_k+X_0$. Plus g\'en\'eralement, notons que les m\'ethodes d\'evelopp\'ees dans le pr\'esent article ne permettent pas de traiter le cas o\`u le polyn\^ome est bilin\'eaire en $X_0,X_k$, ni le cas $a(n) = \epsilon_{T_q(n)-2}(n)\epsilon_2(n)$, o\`u l'indice d\'epend \'egalement de $n$.

\section{Petite propagation}\label{section petite propa}
Ici, et désormais, nous fixons $q$ et $\beta$ des entiers sup\'erieurs ou \'egaux \`a $2$. Le but de cette partie est de d\'emontrer que les fonctions associ\'ees aux suites $\beta$-r\'ecursives v\'erifient la faible propri\'et\'e de petite propagation. L'id\'ee principale consiste \`a exploiter la Proposition \ref{Decomposition pour une suite beta recursive} pour dire que sous certaines conditions, il n'y a pas de diff\'erence entre $f(n)$ et $f^{(\lambda)}(n)$. 

\begin{proposition}\label{Prop petite propa suite type RS}
Soit $(a(n))_{n \in \N}$ une suite $\beta$-r\'ecursive, et $f$ sa fonction associ\'ee. Alors  $f$ a la faible propri\'et\'e de petite propagation.
\end{proposition}

Nous allons tout d'abord d\'emontrer un r\'esultat interm\'ediaire : 

\begin{proposition}\label{Prop pour contraposee petite propa}
Soient $(a(n))_{n \in \N}$ une suite $\beta$-r\'ecursive, et $f$ sa fonction associ\'ee. Soient $(\lambda, \kappa, \rho) \in \mathbb{N}^3$ avec $\rho < \lambda$ et $ \kappa \geq 1$. Soient des entiers $l > q^{\rho}$ et $k_1 < q^{\kappa}$. Supposons qu'il existe un entier $m$ tel que $0 \leq m < \rho - \beta + 2 $ avec $\epsilon_{\kappa + m}(lq^{\kappa}+k_1)\neq q-1$ et que, si on note $i$ le plus petit de ces $m$, il existe un entier $j$ v\'erifiant $i + \beta - 2 < j < \rho $ et $\epsilon_{\kappa + j}(lq^{\kappa}+k_1) \neq 0$. Alors pour tout entier $k_2 < q^{\kappa}$ : \begin{equation*}
f(lq^{\kappa}+k_1+k_2)\overline{f(lq^{\kappa}+k_1)} = f^{(\kappa + \rho)}(lq^{\kappa}+k_1+k_2)\overline{f^{(\kappa + \rho)}(lq^{\kappa}+k_1)}.
\end{equation*}  
\end{proposition}

\begin{proof}
On note $n_1 = lq^{\kappa}+k_1$, $n_1' = n_1 \, \text{mod} \, q^{\kappa+\rho}, N_1 = T_q(n_1)$ et $N_1' = T_q(n_1')$, autrement dit $N_1$ (respectivement $N_1'$) est l'emplacement du dernier chiffre non nul de $n_1$ (respectivement $n_1'$).

On remarque que pour tout $0 \leq k < \kappa + \rho$ : $\epsilon_k(n_1) = \epsilon_k(n_1') $. Comme il existe un entier $j$ tel que  $\beta - 2 < j < \rho$ et $\epsilon_{\kappa + j}(n_1) \neq 0 $ (car $i\geq 0$), on obtient $N_1 \geq N_1' \geq \kappa + j > \kappa + \beta - 2 \geq \beta - 2 $. Donc $N_1 \geq N_1' \geq \beta - 1 $, et la Proposition \ref{Decomposition pour une suite beta recursive} s'applique, si bien que : \!
\begin{align}\label{equa n1}
a(n_1) = & a\bigg(\varphi\Big(\epsilon_{N_1}(n_1)\cdot\epsilon_{N_1-1}(n_1)\cdot\ldots\cdot\epsilon_{N_1-\beta+2}(n_1)\Big) \bigg)
\\ \notag & +\ds\sum_{l=0}^{N_1-\beta+1}g\Big(\epsilon_{l+\beta-1}(n_1)\cdot\ldots\cdot\epsilon_{l+1}(n_1)\cdot\epsilon_l(n_1)\Big),
\end{align}
et on a une formule similaire pour $n_1'$ (avec $N_1'$ au lieu de $N_1$).
On pose \`a pr\'esent $n_2 = n_1 + k_2$, $n_2' = n_2 \, \text{mod} \, q^{\kappa + \rho}$ et $N_2 = T_q(n_2)$ et $N_2' = T_q(n_2')$ leurs tailles respectives. Alors pour tout $k > i$, \begin{equation}\label{equation n1 = n2}
\epsilon_{\kappa + k}(n_1) = \epsilon_{\kappa + k}(n_2).
\end{equation} En effet, il ne peut y avoir de diff\'erence dans les chiffres d'indices supérieurs ou égaux à $k$ de $n_1$ et $n_2$ que dans le cas d'une propagation sur les chiffres de $n_1$. Or, si on veut une propagation jusqu'au chiffre $\kappa + r$, il faut que les chiffres compris entre $\kappa$ et $\kappa + r -1$ de $n_1$ soient tous \'egaux \`a $q-1$. Ainsi, par hypoth\`ese, une propagation \'eventuelle s'arr\^ete \`a $\kappa + i$. 

Comme $N_1, N_1' \geq \kappa + j > \kappa + i + \beta - 2\geq \kappa + i$, on a $N_1' = N_2'$ et $N_1 = N_2$, et finalement on a une formule similaire à \eqref{equa n1} pour $n_2$ et $n_2'$. 

En rassemblant ces différentes formes, on obtient : 
\begin{align}
\label{estim a(k) 1}f(n_1)\overline{f(n_2)f(n_1')}f(n_2') = e\Bigg(\alpha  \Bigg( & a\bigg(\varphi\Big(\epsilon_{N_1}(n_1)\cdot\epsilon_{N_1-1}(n_1)\cdot\ldots\cdot\epsilon_{N_1-\beta+2}(n_1)\Big) \bigg) \Bigg.\Bigg.
\\ \label{estim a(k) 2} &  \Bigg. -a\bigg(\varphi\Big(\epsilon_{N_1}(n_2)\cdot\epsilon_{N_1-1}(n_2)\cdot\ldots\cdot\epsilon_{N_1-\beta+2}(n_2)\Big) \bigg) \Bigg.
\\ \label{estim a(k) 3} &  \Bigg.-a\bigg(\varphi\Big(\epsilon_{N_1'}(n_1')\cdot\epsilon_{N_1'-1}(n_1')\cdot\ldots\cdot\epsilon_{N_1'-\beta+2}(n_1')\Big) \bigg)\Bigg.
\\ \label{estim a(k) 4}&  + \Bigg. a\bigg(\varphi\Big(\epsilon_{N_1'}(n_2')\cdot\epsilon_{N_1'-1}(n_2')\cdot\ldots\cdot\epsilon_{N_1'-\beta+2}(n_2')\Big) \bigg)\Bigg.
\\ \label{estim a(k) 5} & + \ds\sum_{l = N_1'-\beta+ 2}^{N_1-\beta+1}\left[g\Big(\epsilon_{l+\beta-1}(n_1)\cdot\ldots\cdot\epsilon_{l+1}(n_1)\cdot\epsilon_l(n_1)\Big)\right.
\\ \notag & \quad \quad \quad \quad \quad \quad \Bigg.\Bigg.\left.-g\Big(\epsilon_{l+\beta-1}(n_2)\cdot\ldots\cdot\epsilon_{l+1}(n_2)\cdot\epsilon_l(n_2)\Big)\right]\Bigg)\Bigg).
\end{align}

Cependant, comme $N_1' \geq \kappa + j > \kappa + i + \beta - 2$, on a $N_1'- \beta + 2 > \kappa +i$, et donc, pour tout $N_1'-\beta+ 2 \leq l \leq N_1$, par (\ref{equation n1 = n2}), $ \epsilon_l(n_1) = \epsilon_l(n_2)$, et donc : $$\ds\sum_{l = N_1'-\beta+ 2}^{N_1-\beta+1}\left[g\Big(\epsilon_{l+\beta-1}(n_1)\cdot\ldots\cdot\epsilon_{l+1}(n_1)\cdot\epsilon_l(n_1)\Big) - g\Big(\epsilon_{l+\beta-1}(n_2)\cdot\ldots\cdot\epsilon_{l+1}(n_2)\cdot\epsilon_l(n_2)\Big)\right] = 0. $$ En appliquant le m\^eme raisonnement pour $(\ref{estim a(k) 1})$ \`a $(\ref{estim a(k) 4}) $, on trouve $f(n_1)\overline{f(n_2)f(n_1')}f(n_2')  = 1, $ ce qui est bien le r\'esultat voulu.
\end{proof}

\begin{proof}[Preuve de la Proposition \ref{Prop petite propa suite type RS}]
Pour commencer, on remarque que si $lq^{\kappa}+2(q^{\kappa}-1) < q^{\kappa + \rho}$, on a toujours $f(lq^{\kappa}+k_1+k_2)\overline{f(lq^{\kappa}+k_1)} = f^{(\kappa + \rho)}(lq^{\kappa}+k_1+k_2)\overline{f^{(\kappa + \rho)}(lq^{\kappa}+k_1)}.$ En effet, comme $k_1, k_2 \in \{0,\ldots,q^\kappa -1\}$, on a toujours $lq^\kappa + k_1 + k_2 < q^{\kappa + \rho}$ et donc $lq^\kappa + k_1 + k_2 = lq^\kappa + k_1 + k_2 \, \text{mod} \, q^{\kappa + \rho}$.

Soit maintenant $l \geq q^{\rho}$. La Proposition \ref{Prop pour contraposee petite propa} donne des conditions \`a v\'erifier pour que (\ref{petite propagation}) ne soit pas r\'ealis\'ee. On peut donc \'ecrire que l'ensemble 
 
\begin{align*}
& \left\{ 0 \leq l < q^{\lambda} :   \exists \, 0 \leq k_1, k_2 < q^{\kappa} : \right.
\\ &
 \left. \quad f(lq^{\kappa}+k_1+k_2)\overline{f(lq^{\kappa}+k_1)} \neq f^{(\kappa + \rho)}(lq^{\kappa}+k_1+k_2)\overline{f^{(\kappa + \rho)}(lq^{\kappa}+k_1)} \right\}
\end{align*}

est inclus dans l'union $A\cup B\cup C$ avec 
\begin{align*}
A:= &  \left\{ q^{\rho} \leq l < q^{\lambda} :\; \exists \, 0 \leq k_1 < q^{\kappa} : \forall \, 0 \leq i < \rho - \beta + 2, \, \epsilon_{\kappa + i}(lq^{\kappa}+ k_1) = q-1 \right\}, 
\\B:=&  \left\{ q^{\rho} \leq l < q^{\lambda} :\; \exists \, 0 \leq k_1 < q^{\kappa} : \exists \, 0 \leq i < \rho- \beta + 2, \, \epsilon_{\kappa + i}(lq^{\kappa}+ k_1) \neq q-1, \right. 
 \\ &\left.  \quad \forall  m < i, \epsilon_{\kappa + m}(lq^{\kappa}+k_1) = q-1 , 
  \quad   \forall  j \,: \, i + \beta - 2 < j < \rho , \quad \epsilon_{\kappa + j}(lq^{\kappa}+ k_1) = 0 \right\} ,\\
C:= & \left\{l :\; lq^{\kappa} + 2\left(q^{\kappa} - 1 \right) \geq q^{\kappa+ \rho}, \quad lq^{\kappa} \leq q^{\kappa + \rho} \right\}.
\end{align*}
Cependant $lq^{\kappa} \leq q^{\kappa + \rho} $ implique $l \leq q^{\rho} $, et $(q^{\rho}-2)q^{\kappa} + 2(q^{\kappa}-1) = q^{\kappa + \rho} - 2 < q^{\kappa + \rho}$ implique $l \geq q^{\rho}-1 $. On en d\'eduit que $C = \{q^{\rho}-1, q^{\rho} \}$.\\
Il nous reste donc \`a \'evaluer les cardinaux de $A$ et $B$. Pour ce faire on remarque que, quel que soit $k_1 < q^{\kappa}, \epsilon_{\kappa + i}(lq^{\kappa}+k_1) = \epsilon_i(l)$, ce qui nous permet de dire que : 
\begin{align*}
A &=   \{ q^{\rho} \leq l < q^{\lambda} :\; \exists \, 0 \leq k_1 < q^{\kappa} :\; \forall \, 0 \leq i < \rho - \beta + 2, \; \epsilon_{\kappa + i}(lq^{\kappa}+ k_1) = q-1 \} 
\\ &=  \{ q^{\rho} \leq l < q^{\lambda} :\; \forall \, 0 \leq i < \rho-\beta+2, \; \epsilon_{i}(l) = q-1 \},
\end{align*}
donc
$$\# A= \frac{q^{\lambda} - q^{\rho}}{q^{\rho-\beta+2}} = q^{\beta - 2}\left(q^{\lambda - \rho}-1\right).$$
On peut d'autre part \'ecrire $B=\cup_{i = 0}^{\rho - \beta+1}B_i $  avec
\begin{align*}
B_i &=    \left\{ q^{\rho} \leq l < q^{\lambda} :\; \exists \, 0 \leq k_1 < q^{\kappa} :  \; \epsilon_{\kappa + i}(lq^{\kappa}+ k_1) \neq q-1, \quad   \forall  m < i, \epsilon_{\kappa + m}(lq^{\kappa}+k_1) = q-1 , \right. 
\\ & \quad \quad \quad \quad \left.   \text{et pour tout} \;  j \; \text{tel que} \; i + \beta - 2 < j < \rho ,\; \epsilon_{\kappa + j}(lq^{\kappa}+ k_1) = 0 \right\}.
\end{align*}
Mais comme tous les $B_i$ sont en bijection entre eux, on a $\#B = (\rho - \beta+2) \#B_0$. \\
Enfin, on a 
$$\#B_0 = \#  \{ q^{\rho} \leq l < q^{\lambda} :  \forall j : \, \beta - 2 < j < \rho, \, \epsilon_{j}(l) = 0 \} 
= \frac{q^{\lambda} - q^{\rho}}{q^{\rho -\beta + 1}}  = q^{\beta-1}\left(q^{\lambda - \rho} - 1\right).$$

En mettant les trois estimations ensemble, on trouve  :\begin{align*}
\# & \left\{ 0 \leq l < q^{\lambda} : \exists \, 0 \leq k_1, k_2 < q^{\kappa} : \right.
\\ & \quad \left.  f(lq^{\kappa}+k_1+k_2)\overline{f(lq^{\kappa}+k_1)} \neq f^{(\kappa + \rho)}(lq^{\kappa}+k_1+k_2)\overline{f^{(\kappa + \rho)}(lq^{\kappa}+k_1)} \right\}
\\ & \leq q^{\beta - 2}\left(q^{\lambda - \rho}-1\right)+(\rho - \beta+2)q^{\beta-1}\left(q^{\lambda - \rho}-1\right) + 2\ll q^{\beta }(q^{\lambda - \rho + \log \rho}).
\end{align*}
Comme $q^\beta$ est une constante ne d\'ependant que de $q$, la fonction est bien de faible petite propagation.

\end{proof}

\begin{remark}
Dans~\cite{RudinShapiro}, Mauduit et Rivat regardent le cas particulier de la suite Rudin-Shapiro. Pour cette suite, la d\'ecomposition de la Proposition \ref{Decomposition pour une suite beta recursive} se fait automatiquement car 
\begin{enumerate}[label = (\Alph *)]
\item $a(k) = 0$ pour tout $0 \leq k < q$;
\item \label{condi RS pour troncature chez MR} $g(a\cdot b) \neq 0 \Leftrightarrow a = b = 1$.
\end{enumerate}
Ainsi on peut \'ecrire, en notant $N = T_q(n) $ et $N_{\lambda} = T_q(n \, \bmod q^\lambda) $ :
$$a(n)-a^{(\lambda)}(n) = a(\epsilon_N(n))-a(\epsilon_{N_\lambda}(n))+\ds\sum_{i = N_\lambda}^{N-1}g(\epsilon_{i+1}(n)\cdot\epsilon_i(n))
= \ds\sum_{i = \lambda}^{N-1}g(\epsilon_{i+1}(n)\cdot\epsilon_i(n))$$
car on sait qu'alors, pour tout $N_{\lambda} < i < \lambda, \, \epsilon_i(n) = 0 $, et donc $g(\epsilon_{i+1}(n)\cdot \epsilon_i(n)) = 0 $, en vertu de \ref{condi RS pour troncature chez MR}. Ceci permet alors de dire, en reprenant les notations de la d\'emonstration de la Proposition \ref{Prop pour contraposee petite propa}, que 
\begin{equation}\label{equation petite propa RS}
a(n_1)-a(n_2)-a(n_1')+a(n_2') = \ds\sum_{i = \lambda}^{N-1}g(\epsilon_{i+1}(n_1)\cdot\epsilon_i(n_1)) - \ds\sum_{i = \lambda}^{N-1}g(\epsilon_{i+1}(n_2)\cdot\epsilon_i(n_2)),
\end{equation}
et pour s'assurer de la nullit\'e de (\ref{equation petite propa RS}), il suffit de s'assurer que $\epsilon_i(n_1) = \epsilon_i(n_2) $ d\`es que $i$ d\'epasse $\lambda$. Si on suppose uniquement l'existence d'un chiffre d'indice $m < \lambda$, $\epsilon_m(n_1) \neq q-1$, alors cette condition est assur\'ee (car la propagation ne pourra se faire au del\`a du $m$, et on a effectivement $\lambda > m$). 

\smallskip

Dans le cas général, ce raisonnement ne tient plus, et nous sommes obligés d'introduire une fenêtre de sécurité. Il s'agit de la condition sur $j$ dans la Proposition \ref{Prop pour contraposee petite propa}. Cette condition dans la preuve de la Proposition \ref{Prop petite propa suite type RS} entraîne la création de l'ensemble $B$ (l'ensemble $A$ est la contrapos\'ee de la condition sur $m$, et l'ensemble $C$, lui, est un ensemble exceptionnel). Enfin, c'est cet ensemble $B$ qui donne la majoration en $q^{\log \rho} $. 
\end{remark}

Il convient d\'esormais de consid\'erer que les fonctions associ\'ees aux suites $\beta	$-r\'ecursives v\'erifient l'\'equation \begin{equation}\label{equa a verifier pour MR}
\left|\ds\frac{1}{q^N}\ds\sum_{n < q^N}f(n)e(nt)\right| \leq q^{-\gamma(N)},
\end{equation} avec $\gamma(N) \rightarrow \infty$ de mani\`ere croissante. La partie suivante sert \`a introduire des notions qui permettent ce genre de contr\^ole. 

\section{G\'en\'ealogie des fonctions}\label{section genealogie}
Nous commen\c cons par une d\'efinition g\'en\'erale.
\begin{definition}
Soient une application $f : \N \rightarrow \mathbb{U}$ et $\omega$ un mot. On pose \begin{equation}\label{def ancetre}
f_{\omega}(n) := f\big(q^{|\omega|}n + \varphi(\omega)\big).
\end{equation}
\end{definition}

Le lemme suivant permet de donner une formule de r\'ecurrence pour $f_{\omega}(n)$ si $f$ est associ\'ee \`a une suite $\beta$-r\'ecursive.

\begin{lemma}\label{Lemme sur les descendants}
On a
$$ f_\omega (qn+r)= \left \{
\begin{array}{ccc}
    f_{\hat{r}\cdot\omega}(n) & \text{si} &|\omega| < \beta-1; \\
    f_{\hat{r}\cdot\overline{\omega}^{|\omega|-1}}(n)e\big(\alpha g(\hat{r}\cdot \omega ) \big) &\text{si} & |\omega| = \beta-1.  \\
\end{array}
\right.$$
\end{lemma}

\begin{proof}
Par l'\'equation (\ref{def ancetre}) : \begin{align*}
f_\omega (qn+r) &= f\big(q^{|\omega|}(qn+r) + \varphi(\omega)\big)
\\ &= f\big(q^{|\omega|+ 1}n + q^{|\omega|}r + \varphi(\omega)\big) 
\\ &= f\big(q^{|\hat{r}\cdot\omega|}n + \varphi(\hat{r}\cdot\omega)\big). 
\end{align*}

Rappelons que $f(n)= e(\alpha a(n))$, on conclut en utilisant (\ref{def ancetre}) si $|\omega| < \beta-1$  (donc $|\hat{r}\cdot \omega| < \beta $), et en utilisant (\ref{condi 1 beta rec})  ainsi que (\ref{def ancetre})  si $|\omega| = \beta-1$, donc $|\hat{r}\cdot\omega| = \beta$ .
\end{proof}

\begin{definition}
On munit $\M^*$ de l'ordre $\preceq$ suivant. Si $|\omega| \leq |\omega'|$, alors $\omega \preceq \omega'$. Si les deux tailles sont égales, on compare les deux mots par leur ordre lexicographique lu de gauche à droite.\footnote{Ainsi $00 \preceq 01 \preceq 10 \preceq 000$.
} Si $\psi$ d\'esigne la fonction qui \'enum\`ere $\M^*$, alors on d\'efinit $\phi : \mathbb{N} \rightarrow \M^*$ par $\phi = \psi^{-1}$.
\end{definition}

Le but de cette partie est d'exploiter la structure des suites $\beta$-r\'ecursives afin de montrer qu'elles satisfont la d\'efinition \ref{propriete fourier mauduit rivat}. Pour ce faire, nous exploitons le Lemme \ref{Lemme sur les descendants}.

\begin{definition}
On d\'efinit le vecteur $V_n$ de taille $(q^{\beta}-1)/(q-1)$ par $$V_n[l] = f_{\phi(l)}(n),\qquad 0\leq l \leq (q^{\beta}-1)/(q-1)-1.$$ On dira que $V_n$ est le $n$-i\`eme vecteur g\'en\'ealogique de $f$.

\end{definition}

Par le Lemme \ref{Lemme sur les descendants}, il existe une matrice $M_l(\alpha,t)$ telle que $V_{qn+l}e((qn+l)t) = M_l(\alpha,t)V_ne(qnt),$ et donc si on note $S(N,t) := \ds\sum_{n < q^N}V_ne(nt),$ on peut alors \'ecrire : 
\begin{equation*}
S(N,t)  = \sum_{0 \leq n < q^{N-1}}\sum_{0 \leq l < q}V_{qn+l}e((qn+l)t) = \sum_{0 \leq n < q^{N-1}}\sum_{0 \leq l < q}M_l(\alpha,t)V_ne(qnt) =  M(\alpha,t)S(N-1,qt),
\end{equation*}
o\`u on a pos\'e $M(\alpha,t) = \ds\sum_{0 \leq l < q}M_l(\alpha,t)$. En itérant $\beta$ fois et en posant $\widetilde{M}(\alpha,t)=\ds\prod_{0 \leq k < \beta}M(\alpha,q^kt)$, on obtient $S(N,t) = \widetilde{M}(\alpha,t)S(N-\beta,q^{\beta}t) $. On dit que $\widetilde{M}(\alpha,t)$ est la matrice g\'en\'ealogique de $f$. En continuant le raisonnement ci-dessus, on obtient : 
\begin{align}\label{equation vecteur matrice}
\left|\!\left|\sum_{0 \leq n < q^N}V_ne(nt)\right|\!\right|_{\infty} & \leq \ds\prod_{i = 0}^{\lfloor N/\beta \rfloor-1}\left|\!\left|\widetilde{M}(\alpha,q^{i\beta}t)\right|\!\right|_{\infty}\ds\sum_{n < q^{N \text{mod}\beta}}\left|\!\left|V_ne(q^{\beta \lfloor N/\beta \rfloor}nt)\right|\!\right|_{\infty}
\\ \notag & \leq \ds\prod_{i = 0}^{\lfloor N/\beta \rfloor-1}\left|\!\left|\widetilde{M}(\alpha,q^{i\beta}t)\right|\!\right|_{\infty}q^{N \, \text{mod} \, \beta},
\end{align}

où $N \bmod \beta$ désigne le reste de la division euclidienne de $N$ par $\beta$.

\smallskip

La proposition suivante est destin\'ee \`a faire le lien entre la matrice g\'en\'ealogique et l'estimation (\ref{equa a verifier pour MR}).

\begin{proposition}\label{prop propagation vecteur}

Pour tout  entier $\kappa \geq 0$, on a : \begin{equation*}
 \left|\ds\frac{1}{q^N}\ds\sum_{0 \leq n < q^N}f(q^{\kappa}n)e(-nt)\right| \leq \ds\frac{1}{q^{\beta \lfloor N/\beta \rfloor}} \ds\prod_{i = 0}^{\lfloor N/\beta \rfloor - 1}\left|\!\left|\widetilde{M}(\alpha,q^{i\beta}t)\right|\!\right|_{\infty}.
\end{equation*}

\end{proposition}

\begin{proof}

Pour tout entier $\kappa \geq \beta$ , on a $$\left|\ds\sum_{n < q^N}f(q^\kappa n)e(nt)\right|  = \left|\ds\sum_{n < q^N}f(q^{\beta - 1}n)e(-nt)\right|.  $$
En effet, une  r\'ecurrence montre que, pour tout $0 \leq r \leq \kappa - \beta + 1 : a(q^{\kappa}n) = a(q^{\kappa-r}n)+rg(0\cdot0\cdot\ldots\cdot0) $ o\`u $0\cdot0\cdot\ldots\cdot0$ est de taille $\beta$. 

Plus précisément, si $a(q^{\kappa}n) = a(q^{\kappa - r}n)+rg(0\cdot\ldots\cdot0)$ , alors
\begin{align*}
a(q^\kappa n) &= a(q^{\kappa-r}n+ \varphi(0\cdot0\cdot\ldots \cdot 0))+rg(0\cdots0)
\\ &= a(q^{\kappa-r-1}n+ \varphi(0\cdots0))+(r+1)g(0\cdots0)
\\ &= a(q^{\kappa-r-1}n)+(r+1)g(0\cdots0).
\end{align*}
En particulier $a(q^{\kappa}n) = a(q^{\beta-1}n)+(\kappa-\beta+1)g(0\cdots0) $. De plus, si  $\kappa < \beta$, $f(q^{\kappa}n)$ est la $(q^{\kappa}-1)/(q-1)$-i\`eme coordonn\'ee de $V_n$ . Nous concluons la preuve par l'équation (\ref{equation vecteur matrice}) en observant que $N - N \bmod \beta = \beta \lfloor N/\beta \rfloor $.

\end{proof}

D'apr\`es la Proposition \ref{prop propagation vecteur}, il est d\'esormais important d'avoir un contr\^ole sur la norme infini de la matrice $\widetilde{M}(\alpha , t)$. C'est l'objet du résultat suivant.

\begin{proposition}\label{prop equa a prouver thm principal}
\begin{equation}\label{equa a prouver thm principal}
|\!| \widetilde{M}(\alpha,t) |\!|_{\infty} =  \ds\sup_{\gamma \in \M_{\beta \! - \! 1}^*}\,\ds\sum_{\omega \, \in \, \M_{\beta \! - \! 1}}\left|\ds\sum_{k \in \M_1}e\left(t(\hat{k}+q\varphi(\omega))+\alpha\left(\ds\sum_{m \leq |\gamma|}g(\underline{\omega}_{|\omega|-m}\cdot k\cdot \overline{\gamma}^m)\right)\right)\right|. 
\end{equation} 
\end{proposition}

\begin{proof}
Soit $\mathbb{G} $ le graphe suivant : 

\bigskip

\begin{center}

\begin{tikzpicture}[->,>=stealth',shorten >=0.5pt,on grid, auto, node distance=1.4cm,
                    semithick]

  \tikzstyle{every state}=[fill=white,draw=none,text=black]

  \node[state] (A)                    {$f_{\epsilon}$};
  \node[state]         (B) [below of=A] {$...$};
  \node[state]         (D) [ right of=B] {$f_{q-1}$};
  \node[state]         (O) [ left of = B] {$f_{0}$};
  \node[rectangle](DB) [below of = O]{$f_{(q-1)0} $};
  \node[state](DC) [left of = DB]{$... $};
  \node[state](DD) [left of = DC]{$f_{00} $};
  \node[rectangle](BD) [below of = D]{$f_{0(q-1)} $};
  \node[state](BC) [right of = BD]{$... $};
  \node[rectangle](BB) [right of = BC]{$f_{(q-1)(q-1)} $};
  \node[state](interm) [below of = B]{$... $};
  \node[state](interm2) [below of = interm]{$...$};
  \node[state](interm final) [below of = interm2]{$...$};
  \node[state](interm final2) [left of = interm final]{$...$};
  \node[state](interm final3) [right of = interm final]{$...$};
   \node[state](vide) [below of = DD]{$... $ };
  \node[state](famille a gauche c) [below of = vide]{$ ...$};
   \node[state](famille a gauche d) [right of = famille a gauche c]{$ f_{(q-1)0\cdots0}$};
    \node[state](famille a gauche g) [left of = famille a gauche c]{$ \quad f_{0\cdots0}$};
 \node[state](vide 2) [below of = BB]{$...  $};
 \node[state](famille a droite c) [below of = vide 2]{$... $};`
 \node[rectangle](famille a droite g)[left of = famille a droite c]{$f_{0(q-1)\cdots(q-1)} $};
  \node (famille a droite d)[right of = famille a droite c]{$f_{(q-1)\cdots(q-1)} $};
 
  \path[->]
       (A)       edge  (B)
                 edge  (D)
                 edge  (O)
       (O)       edge  (DB)
                 edge  (DC)
                 edge  (DD)
       (D)       edge  (BD)
                 edge  (BC)
                 edge  (BB)
       (interm)  edge  (interm2)
       (interm2) edge  (interm final)
       (DD)      edge  (vide)
       (vide)    edge  (famille a gauche d)
                 edge  (famille a gauche g)
       (BB)      edge  (vide 2)
       (vide 2)  edge  (famille a droite d)
                 edge  (famille a droite g)        
       (famille a gauche g)  edge [loop left]  (famille a gauche g)
       (famille a droite d)  edge [loop right]  (famille a droite d);
       
\draw[->,>=latex] (famille a gauche g) to[bend right] (famille a gauche c);
\draw[->,>=latex] (famille a gauche g) to[bend right] (famille a gauche d);
\draw[->,>=latex] (famille a gauche d) to[bend right] (interm final3);
\end{tikzpicture}

\end{center}

\bigskip

 $\mathbb{G}$ repr\'esente la mani\`ere dont peut \'evoluer en $k$ \'etapes ($k$ arbitraire) un mot $ \gamma$ donn\'e selon le Lemme \ref{Lemme sur les descendants}, d\'ecrivons le.
 
$\mathbb{G} $ est un graphe descendant qui poss\`ede $\beta$ lignes. \`A chaque fl\`eche correspond une alt\'eration de l'argument. Chaque \'el\'ement donne $q$ descendants, et donc le graphe poss\`ede $\frac{q^{\beta}-1}{q-1}$ sommets. Si on se trouve sur la derni\`ere ligne, avec un mot $\gamma$, un descendant $\gamma'$ de $\gamma $ sera donn\'e par $\gamma' = \overline{\gamma'}^{1}\cdot\overline{\gamma}^{|\gamma|-1} $.  On se prom\`ene sur le graphe en respectant les r\`egles suivantes.\begin{enumerate}[label = (\roman *)]
\item On suit le sens des fl\`eches.
\item \label{Regle si taille petite} Si \`a la $k$-i\`eme \'etape, on passe d'un mot $\gamma $ \`a un mot $\gamma '$, avec la taille de $\gamma$ strictement plus petite que $\beta - 1$, on ajoute $q^{k-1}\overline{\gamma'}^{1}t $ \`a l'argument. 
\item \label{Regle si taille max} Si \`a la $k$-i\`eme \'etape, on passe d'un mot $\gamma $ \`a un mot $\gamma '$, avec la taille de $\gamma$ \'egale \`a $\beta - 1$, on ajoute $q^{k-1}\overline{\gamma'}^{1}t + \alpha g(\overline{\gamma'}^{1}\cdot\gamma) $ \`a l'argument.
\end{enumerate} 
D\'esormais, nous appelons encodage d'un chemin la valeur $e(x)$, où $x$ est l'argument total du chemin lorsque ce dernier est soumis aux r\`egles ci-dessus.

Soit \`a pr\'esent $\text{Enc}_k(\gamma , \omega)$,  la somme des encodages concernant tous les chemins possibles en $k$ \'etapes reliant $\gamma$ \`a $\omega$. 

\begin{example}
Le graphe suivant correspond au cas $q= 2$, $\beta = 3$.

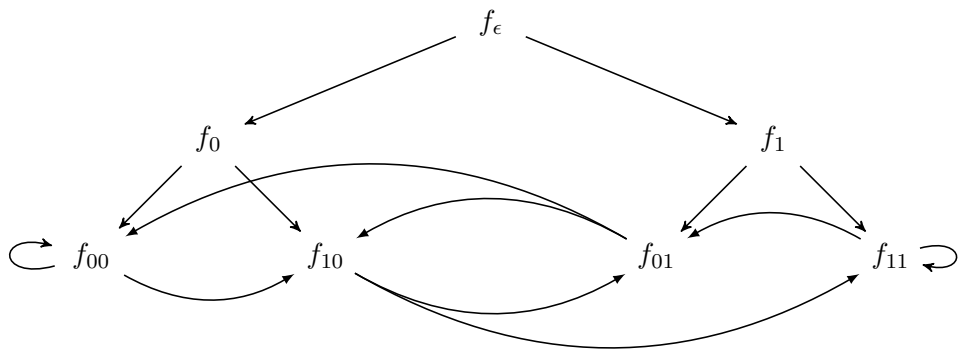
\begin{center}

\begin{tikzpicture}[->,>=stealth',shorten >=0.5pt,on grid, auto, node distance=2.2cm,
                    semithick]

  \tikzstyle{every state}=[fill=white,draw=none,text=black]

  \node[state] (A)                    {$f_{\epsilon}$};
  \node[state]         (D) [below right of=A] {};
  \node[state]         (M) [right of = D] {$f_1 $};
  \node[state]         (O) [below left of = A] {};
  \node[state]         (N) [left of = O] {$f_0 $};
  \node[rectangle](DB) [below right of = N]{$f_{10} $};
  \node[state](DD)     [below left of = N]{$f_{00} $};
  \node[rectangle](BD) [below left of = M]{$f_{01} $};
  \node[rectangle](BB) [below right of = M]{$f_{11} $};
  
  \path[->]
       (A)       edge  (M)
                 edge  (N)
       (N)       edge  (DB)
                 edge  (DD)
       (M)       edge  (BD)
                 edge  (BB)
       (DD)  edge [loop left]  (DD)
       (BB)  edge [loop right]  (BB);
  
  \draw[->,>=latex] (DD) to[bend right] (DB);
  \draw[->,>=latex] (BD) to[bend right] (DB);
  \draw[->,>=latex] (DB) to[bend right] (BD);
  \draw[->,>=latex] (BD) to[bend right] (DD);
  \draw[->,>=latex] (DB) to[bend right] (BB);
  \draw[->,>=latex] (BB) to[bend right] (BD);

\end{tikzpicture}

\end{center}

Dans ce graphe, qui correspond à $q=2,\beta = 3$, il y a deux manières d'aller du mot $\epsilon$ au mot $00$ en trois étapes : en faisant le chemin $$\epsilon \underset{0}{\rightarrow} 0 \underset{0}{\rightarrow} 00 \underset{0}{\rightarrow} 00 $$ et en faisant le chemin $$\epsilon \underset{1}{\rightarrow} 1 \underset{0}{\rightarrow} 01 \underset{0}{\rightarrow} 00. $$ Ceci nous donne donc :  $$\text{Enc}_3(\epsilon , 00) = e(\alpha g(000)) + e(t+\alpha g(001)).$$ 

\end{example}

Comme $M_l(\alpha,t) $ est la matrice de passage de $V_n$ \`a $V_{qn+l} $, sommer sur $l$ (c'est à dire regarder $M(\alpha,t)$)  revient alors \`a d\'eterminer tous les chemins \`a une \'etape possible. Et comme $\widetilde{M}(\alpha,t) = \ds\prod_{i < \beta}M(\alpha, q^i t), \widetilde{M}(\alpha,t)[i,j]$ correspond à la somme des encodages concernant tous les chemins possibles en $\beta$ \'etapes reliant $\phi (i)$ \`a $\phi(j)$. .

Nous avons donc : \begin{equation}\label{equa thm etape def norme infini}
|\!| \widetilde{M}(\alpha,t) |\!|_{\infty} =  \sup_{i}\,\ds\sum_{j}\left|\text{Enc}_{\beta}(\phi (i) , \phi (j))\right|.
\end{equation}
Cependant $\phi( i)$ et $\phi (j)$ parcourent l'ensemble des mots de taille au plus $\beta - 1$. Ainsi (\ref{equa thm etape def norme infini}) se transforme en : 
\begin{equation}\label{equa thm etape bijection avec mots}
|\!| \widetilde{M}(\alpha,t) |\!|_{\infty} =  \ds\sup_{\gamma \in \M_{\beta \! - \! 1}^*}\,\ds\sum_{\omega \, \in \, \M_{\beta \! - \! 1}^*}\left|\text{Enc}_{\beta}(\gamma , \omega )\right|.
\end{equation}
Cependant, par le Lemme \ref{Lemme sur les descendants}, pour tout $\gamma$, un descendant de $\gamma$ \`a la $\beta$-ieme g\'en\'eration est forc\'ement de taille $\beta - 1$. En effet : la taille du mot va en croissant, strictement si la taille est strictement plus petite que $\beta-1$, et devient constante d\`es que cette taille est atteinte. Or cette taille est atteinte, au pire, au bout de la $\beta-1$ i\`eme \'etape. Donc  (\ref{equa thm etape bijection avec mots}) se transforme en : 
\begin{equation}\label{equa thm etape restriction taille mots}
|\!| \widetilde{M}(\alpha,t) |\!|_{\infty} =  \ds\sup_{\gamma \in \M_{\beta \! - \! 1}^*}\,\ds\sum_{\omega \, \in \, \M_{\beta \! - \! 1}}\left|\text{Enc}_{\beta}(\gamma , \omega )\right|.
\end{equation}
Il reste donc \`a comprendre $\text{Enc}_{\beta}(\gamma , \omega )$. 

\smallskip

Soit $\gamma \in \M_{\beta - 1}^*$. On pose $R \in \M_{\beta - 1 - |\gamma|}$ et $S \in \M_{|\gamma| + 1}$. Les mots $R$ et $S$ interviendront dans le processus pour aller de $\gamma$ \`a $\omega$ et seront d\'etermin\'es ult\'erieurement. Arriver \`a un mot de taille $\beta - 1 $ se fait en $\beta - 1 - |\gamma | $ \'etapes, c'est \`a dire par l'adjonction de $R$. Nous avons donc, en suivant \ref{Regle si taille petite}, le chemin suivant : 
\begin{equation}\label{equa chemin arriver a taille}
\gamma \, \underset{\epsilon_0(R) t }{ \longrightarrow} \, \underline{R}_{1}  \cdot\gamma \rightarrow \quad \ldots \quad \rightarrow \underline{R}_{i-1}\cdot\gamma \, \underset{\epsilon_{i-1}(R)q^{i-1} t  }{ \longrightarrow} \, \underline{R}_{i}\cdot\gamma \rightarrow \quad \ldots \quad \rightarrow  \underline{R}_{|R|-1}\cdot\gamma \, \underset{\epsilon_{|R|-1}(R)q^{|R|-1} t  }{ \longrightarrow} \, R\cdot\gamma  .
\end{equation}

Il nous reste $\beta - (\beta - 1 - |\gamma|) = |\gamma| + 1$ \'etapes \`a parcourir. C'est \`a dire \`a concat\'ener $S$. Cependant comme on a atteint un mot de taille $\beta - 1 $, la fonction de propagation $g$ s'adjoint \`a l'argument (il s'agit de la règle  \ref{Regle si taille max}). Nous avons donc, en suivant \ref{Regle si taille max}, la cha\^ine suivante (on ajoute \`a l'argument ce qui est en bas de la fl\`eche) : 
\begin{align}\label{equa chemin compression}
& R\cdot\gamma \, \underset{\epsilon_0(S)q^{|R|} t + \alpha g(\underline{S}_1\cdot \overline{R \cdot \gamma}^{|R\cdot \gamma|} )  }{\longrightarrow} \, \underline{S}_{1}\cdot \overline{R\cdot\gamma}^{|R\cdot\gamma|-1}
\\ \notag &  \ldots  
\\ \notag  &\underline{S}_{i}\cdot \overline{R\cdot\gamma}^{|R\cdot\gamma|-i}  \, \underset{\epsilon_i(S)q^{i+|R|} t + \alpha g(\underline{S}_{i+1}\cdot \overline{R\cdot\gamma}^{|R\cdot\gamma|-i}  ) }{ \longrightarrow}  \, \underline{S}_{i+1}\cdot \overline{R\cdot\gamma}^{|R\cdot\gamma|-i-1} 
\\ \notag &  \ldots 
\\ \notag &\underline{S}_{|\gamma|}\cdot \overline{R\cdot\gamma}^{|R\cdot \gamma|-|\gamma|}  \, \underset{\epsilon_{|\gamma|}(S)q^{|\gamma|+|R|} t + \alpha g(\underline{S}_{|\gamma|+1}\cdot \overline{R\cdot\gamma}^{|R\cdot \gamma|-|\gamma|} ) }{ \longrightarrow}  \, S\cdot \overline{R\cdot\gamma}^{|R\cdot \gamma|- |\gamma|-1} = S\cdot \overline{R}^{|R|-1} = \omega .
\end{align}

De la derni\`ere ligne on conclut que $S\cdot R = \omega \cdot k$, avec $k$ un mot de taille $1$. Il suit donc que :
\begin{align*}
\underline{S}_{i+1}\cdot \overline{R\cdot\gamma}^{|R\cdot \gamma|-i} & = \underline{S}_{i+1}\cdot R\cdot \overline{\gamma}^{|\gamma|-i}
\\ & = \underline{S\cdot R}_{|R|+i+1}\cdot \overline{\gamma}^{|\gamma|-i} 
\\ &= \underline{\omega\cdot k}_{\beta - (|\gamma|+1)+i+1} \cdot \overline{\gamma}^{|\gamma|-i}
\\ &= \underline{\omega\cdot k}_{|\omega|+1 - (|\gamma|+1)+i+1} \cdot \overline{\gamma}^{|\gamma|-i}
\\ &= \underline{\omega}_{|\omega| - |\gamma|+i}\cdot k \cdot \overline{\gamma}^{|\gamma|-i} .
\end{align*}

Comme $0 \leq i \leq |\gamma|  $, on a 
\begin{align}\label{equa chgmt variable en i}
\{\underline{S}_{i+1}\cdot \overline{R\cdot\gamma}^{|R\cdot \gamma|-i}, \, 0 \leq i \leq |\gamma| \} &= \{\underline{\omega}_{|\omega| - |\gamma|+i}\cdot k \cdot \overline{\gamma}^{|\gamma|-i}, \, 0 \leq i \leq |\gamma|  \} \nonumber\\
&= \{\underline{\omega}_{|\omega| - i}\cdot k \cdot \overline{\gamma}^i, \, 0 \leq i \leq |\gamma| \}.
\end{align} 

En r\'eunissant (\ref{equa chemin arriver a taille}) et (\ref{equa chemin compression}), et en utilisant (\ref{equa chgmt variable en i}), nous obtenons qu'un encodage, suivant le chemin $ S\cdot R$ est \'egal \`a \begin{align*}
e\left(t\left(  \ds\sum_{i = 0}^{|R|-1}\epsilon_i (R)q^i + q^{|R|} \right.\right.& \left.\left.  \ds\sum_{i = 0}^{|S|-1}\epsilon_i(S)q^i\right) +\alpha\ds\sum_{m = 0}^{|\gamma|}g\left(\underline{S}_{m+1}\cdot \overline{R\cdot\gamma}^{|R\cdot \gamma|-m} \right)\right)
\\ & = e\left(  t\varphi (S\cdot R) +\alpha\ds\sum_{m = 0}^{|\gamma|}g\left(\underline{S}_{m+1}\cdot \overline{R\cdot\gamma}^{|R\cdot \gamma|-m}\right)\right)
\\ & = e\left(  t\varphi (\omega\cdot k) +\alpha\ds\sum_{m = 0}^{|\gamma|}g\left(\underline{\omega}_{|\omega| - m}\cdot k \cdot \overline{\gamma}^m\right)\right)
\\ & = e\left(  t\left(\hat{k} + q \varphi (\omega)\right) +\alpha\ds\sum_{m \leq |\gamma|}g\left(\underline{\omega}_{|\omega| - m}\cdot k \cdot \overline{\gamma}^m\right)\right).
\end{align*}  
En utilisant (\ref{equa thm etape restriction taille mots}) et le fait que $k$  prend toutes les valeurs de $\M_1 $, on obtient bien (\ref{equa a prouver thm principal}) . 
\end{proof}

Nous utilisons \`a pr\'esent cette estimation pour obtenir un contr\^ole uniforme en $t$ de la norme infini de $\widetilde{M}(\alpha,t) $.

\begin{corollary}\label{Scolie}
Soient $\omega_1 , \omega_2 \in \M_{\beta - 1} $, tels que $\underline{\omega_1}_{(\beta - 2)} = \underline{\omega_2}_{(\beta - 2)} $ mais $\omega_1 \neq \omega_2 $\footnote{Les mots $\omega_1$ et $\omega_2$ diff\`erent de $\epsilon_{\beta-2} $, par exemple $\omega_1 = 10000000$ et $\omega_2 = 00000000 $.} et $k_1,k_2 \in \M_1 : k_1 \neq k_2$. Alors \begin{equation}\label{equation scolie}
|\!| \widetilde{M}(\alpha,t) |\!|_{\infty} \leq q^{\beta} - 8\left(\sin \ds\frac{\pi |\!| \alpha \left(g(\omega_1\cdot k_1) - g(\omega_1 \cdot k_2) - g(\omega_2 \cdot k_1) + g(\omega_2\cdot k_2) \right) |\!|_{\mathbb{Z}}}{4}\right)^2.
\end{equation}
\end{corollary}

Tout d'abord, pr\'esentons un lemme trigonom\'etrique, dont on peut retrouver la d\'emonstration dans~\cite{RudinShapiro}. Nous rappelons que $|\!|x|\!|_{\Z} $ repr\'esente la distance du r\'eel $x$ au plus proche entier.

\begin{lemma}\label{lemme trigo pour scolie}
Soient $x, x', \xi, \alpha $ des nombres r\'eels. Alors : 
\begin{equation}\label{equation trigo}
\left| e(x+\alpha') + e(x) \right| + \left| e(x' + \xi) + e(x') \right| \leq 4 - 8\left( \sin \ds\frac{\pi |\!| \xi - \alpha' |\!|_{\mathbb{Z}}}{4 }  \right)^2.
\end{equation}
\end{lemma}

\begin{proof}[D\'emonstration du Corollaire \ref{Scolie}]
En majorant trivialement (\ref{equa a prouver thm principal}), dans les cas o\`u on a $\omega \neq \omega_1, \omega_2 $, $k \neq k_1, k_2$, nous obtenons : 
$$
|\!| \widetilde{M}(\alpha,t) |\!|_{\infty} =  \ds\sup_{\gamma \in \M_{\beta \! - \! 1}^*}  q^{\beta} - 4 + \ds\sum_{i = 1,2}\left|\ds\sum_{j = 1,2}e\left(t(\hat{k_j}+q\varphi(\omega_i))+\alpha\left(\ds\sum_{m \leq |\gamma|}g(\underline{\omega_i}_{|\omega_i|-m}\cdot k_j\cdot \overline{\gamma}^m)\right)\right) \right| .
$$
On pose alors \begin{gather}\label{x alpha and co pour lemme trigo}\begin{aligned}
x & = t(\hat{k_1}+q\varphi(\omega_1))+\alpha\ds\sum_{m \leq |\gamma|}g(\underline{\omega_1}_{|\omega_1|-m}\cdot k_1\cdot \overline{\gamma}^m) ,\\ 
\alpha' & = t(\hat{k_2}-\hat{k_1})+\alpha\ds\sum_{m \leq |\gamma|}\left(g(\underline{\omega_1}_{|\omega_1|-m}\cdot k_2\cdot \overline{\gamma}^m)-g(\underline{\omega_1}_{|\omega_1|-m}\cdot k_1\cdot \overline{\gamma}^m)\right) , \\ 
x' &= t(\hat{k_1}+q\varphi(\omega_2))+\alpha\ds\sum_{m \leq |\gamma|}g(\underline{\omega_2}_{|\omega_2|-m}\cdot k_1\cdot \overline{\gamma}^m) \\
 \text{et} \quad \xi & = t(\hat{k_2}-\hat{k_1})+\alpha\ds\sum_{m \leq |\gamma|}\left(g(\underline{\omega_2}_{|\omega_2|-m}\cdot k_2\cdot \overline{\gamma}^m)-g(\underline{\omega_2}_{|\omega_2|-m}\cdot k_1\cdot \overline{\gamma}^m)\right).
\end{aligned}
\end{gather} 
si bien que 
\begin{equation}\label{equa scolie triviale}
|\!| \widetilde{M}(\alpha,t) |\!|_{\infty} =  \ds\sup_{\gamma \in \M_{\beta \! - \! 1}^*} \left( q^{\beta} - 4 + |e(x)+e(\alpha'+x)|+|e(x')+e(x'+\xi)|\right).
\end{equation}
Or
 \begin{align}\label{diff arguments}
\xi - \alpha' =  \alpha\left(\ds\sum_{m \leq |\gamma|} \right. & \left. \left(g(\underline{\omega_2}_{|\omega_2|-m}\cdot k_2\cdot \overline{\gamma}^m)-g(\underline{\omega_2}_{|\omega_2|-m}\cdot k_1\cdot \overline{\gamma}^m)\right) \right.
\\ \notag &  \left. - \ds\sum_{l \leq |\gamma|}\left(g(\underline{\omega_1}_{|\omega_1|-l}\cdot k_2\cdot \overline{\gamma}^l)-g(\underline{\omega_1}_{|\omega_1|-l}\cdot k_1\cdot \overline{\gamma}^l)\right)\right).
\end{align}
Et comme  pour tout $1 \leq m \leq |\omega_1| $, $\underline{\omega_1}_{|\omega_1|-m} = \underline{\omega_2}_{|\omega_1|-m} $, le seul terme non nul dans (\ref{diff arguments}) est $m=0$, et donc :
\begin{equation}\label{equa scolie finale}
\xi - \alpha' = \alpha \left(g(\omega_1\cdot k_1) - g(\omega_1 \cdot k_2) - g(\omega_2 \cdot k_1) + g(\omega_2\cdot k_2) \right),
\end{equation} 
ce qui conclut nôtre preuve.
\end{proof}

\section{Preuve du Th\'eor\`eme \ref{notre gros theoreme}}\label{Section TP}

\smallskip

La preuve du th\'eor\`eme \ref{notre gros theoreme} suit de pr\`es la preuve tr\`es technique du r\'esultat de Mauduit et Rivat dans \cite{RudinShapiro}. 
Nous ne pr\'esentons pas ici toute la preuve, mais seulement les \'el\'ements modifi\'es. Nous conseillons donc au lecteur de suivre notre raisonnement en ayant~\cite{RudinShapiro} sous les yeux.

Classiquement, Mauduit et Rivat utilisent d'abord une identit\'e de Vaughan pour ramener l'estimation de la somme impliquant la fonction de von Mangoldt  \`a l'\'evaluation de sommes de type $I$ ($S_I(\vartheta)$) et de type $II$ ($S_{II}(\vartheta)$). Chacune de ces sommes est ensuite s\'epar\'ee en deux parties. Dans la premi\`ere partie, on a remplac\'e la fonction bas\'ee sur les chiffres qui intervient par une version tronqu\'ee de cette fonction. La troncation permet d'obtenir des fonctions p\'eriodiques et d'utiliser l'analyse de Fourier pour \'evaluer ces premi\`eres sommes. Ces estimations ne sont pas alt\'er\'ees par nos modifications. Dans la seconde partie en revanche, on estime la contribution de l'erreur commise lors du remplacement des fonctions par leurs versions tronqu\'ees. Pour cette seconde partie, Mauduit et Rivat utilisent les propri\'et\'es de petite propagation des fonctions qu'ils consid\`erent. Nous avons introduit \`a la d\'efinition \ref{MR def 1} une propri\'et\'e alternative plus faible qui sera v\'erifi\'ee par les fonctions que nous consid\'erons. L'affaiblissement de cette propri\'et\'e am\`enera des modifications dans l'estimations de ces secondes sommes.
 
Pour cette partie, nous notons $y \sim q^k $ pour $q^{k-1}\leq y < q^k $.

\subsection{Sommes de type I}
Soient $M$ et $N$ des entiers, avec $1 \leq M \leq N$ et $M \leq (MN)^{1/3}$. Nous notons $\mu$ et $\nu$ les entiers tels que $T_q(M)+1=\mu$ et $T_q(N) +1= \nu$.
Soit $f$ une application  v\'erifiant les d\'efinitions \ref{MR def 1} et \ref{propriete fourier mauduit rivat}. 
Soit $\vartheta \in \mathbb{R}$, $I(M,N) \subset [0,MN]$ un intervalle. Nous cherchons \`a estimer $$S_I(\vartheta) := \ds\sum_{M/q \leq m < M}\left|\ds\sum_{n : mn \in I(M,N)}f(mn)e(\vartheta mn ) \right|. $$ 

Comme expliqu\'e pr\'ec\'edemment, la somme $S_I(\theta)$ est s\'epar\'ee en deux parties, nomm\'ees $S_{I,1}'(\vartheta ')$ et $S_{I,2}'(\vartheta ')$ (voir equations $(30),(31)$ et $(35)$ de~\cite{RudinShapiro}). Dans la premi\`ere somme, la fonction $f$ est remplac\'ee par sa fonction tronqu\'ee, la seconde somme prend en compte l'erreur engendr\'ee par cette substitution. L'estimation de la premi\`ere somme reste inchang\'ee et on a donc comme Mauduit et Rivat 
\begin{equation}\label{equation support pour SI}
S_I(\theta) \ll q^{\mu+\nu}(\log q^{\mu+\nu})(S_{I,1}'(\theta')+S_{I,2}'(\theta'))\ll q^{\mu+\nu}(\log q^{\mu+\nu})\left(\mu (\log q)^{3/2}q^{\frac{\rho_1}{2}-\gamma(\frac{\mu + \nu}{3})}+S_{I,2}'(\theta')\right),
\end{equation} 
o\`u $\rho_1 $ est un entier v\'erifiant $1 \leq \rho_1 \leq \mu + \nu - \kappa $ avec $\kappa$ un entier tel que $1 \leq \kappa \leq \frac{\mu + \nu}{3} $, param\`etres que l'on optimisera ult\'erieurement.

Pour estimer $S_{I,2}'(\vartheta ')$, on a comme Mauduit et Rivat
\begin{equation}\label{SI2'}
S_{I,2}'(\vartheta ') \ll \ds\sum_{1 \leq d \leq M}\ds\frac1d\left(\frac{\log{q}}{q^{\mu+\nu}}\sum_{\omega \in \mathcal{W}_{\kappa_d}}2^2\right)^{1/2}\ds,
\end{equation}
o\`u $\kappa_d$ est choisi de sorte que $M^2/d^2 \sim q^{\kappa_d}$,  $\mathcal{W}_{\kappa} = \{u+vq^{\kappa}, (u,v) \in \widetilde{\mathcal{W}_{\kappa}} \}$ et  
$\widetilde{\mathcal{W}_{\kappa}}$ d\'esigne l'ensemble des paires d'entiers $(u,v) \in \{0,\ldots, q^{\kappa}-1\}\times\{0, \ldots, q^{\mu+\nu-\kappa}-1\} $ pour lesquelles $$f(u+vq^{\kappa})\overline{f(vq^{\kappa})} \neq f^{(\kappa + \rho_1)}(u+vq^{\kappa})\overline{f^{(\kappa + \rho_1)}(vq^{\kappa})}.$$ 
Il n'est pas \'etonnnant de voir appara\^{\i}tre ici une somme sur $\mathcal{W}_{\kappa}$ puisque cette partie de la somme mesure l'erreur commise en rempla\c{c}ant une fonction par sa troncature.
C'est dans l'estimation du cardinal de $\widetilde{\mathcal{W}_{\kappa}}$ qu'un changement apparait. En utilisant la d\'efinition  \ref{MR def 1}, nous obtenons 
\begin{equation}\label{nouvelle equa W k}
\text{card} \, \widetilde{\mathcal{W}_{\kappa}} \ll q^{\mu + \nu - \rho_1 + \log \rho_1 },
\end{equation}
alors que Mauduit et Rivat avaient 
$\text{card} \, \widetilde{\mathcal{W}_{\kappa}} \ll q^{\mu + \nu - \rho_1 }.$

Finalement en combinant \eqref{equation support pour SI}, \eqref{SI2'} et \eqref{nouvelle equa W k} avec le choix $\rho_1 = \gamma((\mu+\nu)/3) $, nous obtenons :
\begin{equation}\label{Estim SI nouvelle final}
S_I(\vartheta)\ll (\log q)^{5/2}(\mu+\nu)^2q^{\mu+\nu -\frac{\gamma((\mu+\nu)/3)}{2} + \frac{\log (\gamma((\mu+\nu)/3))}{2}},
\end{equation}
ce qui, en utilisant le fait que $\gamma(\lambda) \leq \lambda/2 $ (\'equation ($26$) dans~\cite{RudinShapiro}), donne
\begin{equation*}
S_I(\vartheta)\ll (\log q)^{5/2}(\mu+\nu)^{2 + \log q}q^{\mu+\nu -\frac{\gamma((\mu+\nu)/3))}{2}}.
\end{equation*}


\subsection{Sommes de type II}

Nous reprenons les notations introduites pour les sommes de type $I$; nous supposons de plus $$\frac{1}{4}(\mu+\nu) \leq \mu \leq \nu \leq \frac{3}{4}(\mu+\nu). $$ Nous introduisons   $a_m \in \mathbb{C}$ et $b_n \in \mathbb{C}$ avec $|a_m|,|b_n| \leq 1$. Les sommes de type II sont définies par $$S_{II}(\vartheta):= \ds\sum_{M/q\leq m < M}\ds\sum_{N/q \leq n < N}a_mb_nf(mn)e(\vartheta mn). $$ 

D\`es le d\'ebut de l'estimation de cette somme dans~\cite{RudinShapiro}, une premi\`ere troncation est introduite. On est alors amen\'e \`a majorer le nombre de paires $(m,n) \in \{q^{\mu-1}, \ldots, q^{\mu}-1 \}\times\{q^{\nu-1}, \ldots,  q^{\nu}-1 \}$ telles qu'il existe $ k < q^{\mu+\rho}$ avec $f(mn+k)\overline{f(mn)} \neq f^{(\mu + 2\rho)}(mn+k)\overline{f^{(\mu + 2\rho)}(mn)} $ o\`u $\rho\leq \mu/7$ est un param\`etre que l'on choisira ult\'erieurement. Avec la faible propri\'et\'e de petite propagation, nous obtenons que ce nombre est un $O\left((\log q )q^{\mu + \nu - \rho + \log \rho}\right)$ au lieu de $O(q^{\mu+\nu-\rho}) $ dans~\cite{RudinShapiro} (Lemma 8 de \cite{RudinShapiro}). Ceci conduit \`a la majoration
\begin{equation}\label{equation SII premiere partie modifiee}
|S_{II}(\vartheta)|^4 \ll q^{4(\mu+\nu)- 2 \rho + 2\log \rho} + q^{3(\mu + \nu - \rho)}\ds\sum_{1 \leq r < q^{\rho}}\ds\sum_{1 \leq s < q^{2\rho}}|S_2'(r,s)|,
\end{equation}
o\`u $S_2'(r,s)$ est une somme faisant intervenir la fonction doublement tronqu\'ee $f^{(\mu_1,\mu_2)}(n) := f^{(\mu_2)}(n)\overline{f^{(\mu_1)}(n)}$ avec $\mu_1=\mu-2\rho$ et $\mu_2=\mu+2\rho$.

Dans la somme $S_2'(r,s)$, Mauduit et Rivat remplacent la fonction $f^{(\mu_1,\mu_2)}(n)$ par la quantit\'e $f^{(\mu_1,\mu_2)}(r_{\mu_0,\mu_2}(n))$, o\`u $\mu_0=\mu-2\rho-2\rho'$ avec $0\leq\rho'\leq \rho$ et $r_{\mu_0,\mu_2}(n)$  est l'entier $u_1$ dans l'\'ecriture unique
$$n = u_2q^{\mu_2}+u_1q^{\mu_0}+u_0,\quad 0\leq u_1<q^{\mu_2-\mu_0}, \quad u_2\geq 0, \quad 0\leq u_0<q^{\mu_0}.$$

L'erreur engendr\'ee par cette substitution est contr\^ol\'ee par le cardinal de $\mathcal{E}_{\mu_0,\mu_1,\mu_2}(r,s)$, l'ensemble des paires $(m,n)$, avec $M/q < m \leq M $ et $N/q < n \leq N $ (o\`u $M \sim q^{\mu}, N \sim q^{\nu}$) pour lesquelles
$$ f^{(\mu_1,\mu_2)}(mn+q^{\mu_1}sn+q^{\mu_1}sr)\neq f^{(\mu_1,\mu_2)}(q^{\mu_0}r_{\mu_0,\mu_2}(mn+q^{\mu_1}sn+q^{\mu_1}sr)).$$

L'estimation de ce cardinal fait appel \`a la propri\'et\'e de petite propagation et en utilisant notre propri\'et\'e plus faible, on obtient 
\begin{equation}\label{nouvelle equation rho prime}
\text{card} \, \mathcal{E}_{\mu_0,\mu_1,\mu_2}(r,s) \ll \max(\tau(q),\log q)(\mu + \nu)^{\omega(q)}q^{\mu+\nu-2\rho'+\log\mu_1},
\end{equation}
ce qui implique 
\begin{align}\label{mes termes d'erreurs reunis}
|S_{II}(\vartheta)|^4 & \ll q^{4(\mu+\nu)- 2 \rho + 2\log \rho} + \max(\tau(q),\log q)(\mu + \nu)^{\omega(q)}q^{4(\mu + \nu)-2\rho'+\log (\mu-2\rho)}
\\\notag  &  \quad + q^{3(\mu + \nu - \rho)}\ds\sum_{1 \leq r < q^{\rho}}\ds\sum_{1 \leq s < q^{2\rho}}|S_3(r,s)|,
\end{align}
o\`u $S_3(r,s)$ est une somme dans laquelle la fonction $f^{(\mu_1,\mu_2)}(r_{\mu_0,\mu_2}(n))$ intervient.

\medskip
De mani\`ere \`a introduire des transform\'ees de Fourier de $f^{(\mu_1,\mu_2)}$, Mauduit et Rivat identifient la d\'ecomposition en base $q$ avec un sous ensemble de l'intervalle $[0,1)$ translat\'e sur l'ensemble des entiers ($r_{\mu_0,\mu_2}(n)= u \Leftrightarrow \frac{n}{q^{\mu_2}} \in [\frac{u}{q^{\mu_2-\mu_1}},\frac{u+1}{q^{\mu_2-\mu_1}}) + \mathbb{Z})$). Ils introduisent alors des fonctions indicatrice d'intervalles qu'ils contr\^olent \`a l'aide des polyn\^omes de Vaaler. Ils trouvent une nouvelle d\'ecomposition de $S_3(r,s)$ constitu\'ee du terme principal des polyn\^omes, qu'ils nomment $S_4(r,s)$ et des termes d'erreurs qui sont contr\^ol\'es par les m\'ethodes usuelles et ne sont pas affect\'es par notre modification. Ainsi, nous pouvons \'ecrire 
\begin{equation}\label{S3 a S4}
S_3(r,s) = S_4(r,s) + O(\max (\log q^{\mu-2(\rho+\rho')},\tau(q^{\mu-2(\rho+\rho')}))q^{\mu+\nu-2\rho}).
\end{equation}

Le fait d'avoir introduit les polyn\^omes de Vaaler permet de travailler sur les transform\'ees de Fourier de $g(n) = f^{(\mu_1,\mu_2)}(q^{\mu_0}n)$, si bien que $S_4(r,s)$ s'\'ecrit : \begin{align*}
S_4(r,s) = q^{2(\mu_2-\mu_0)}&\ds\sum_{|h_0|,|h_1|\leq H}a_{h_0}(q^{\mu_0-\mu_2},H)a_{h_1}(q^{\mu_0-\mu_2},H)\ds\sum_{0\leq h_2,h_3 < q^{\mu_2-\mu_0}}e\left(\frac{h_3sr}{q^{\mu_2-\mu_1}}\right)
\\&\times\hat{g}(h_0-h_2)\overline{\hat{g}(h_3-h_1)\hat{g}(-h_2)}\hat{g}(h_3)
\\ &\times \ds\sum_{m,n}e\left(\ds\frac{(h_0+h_1)mn+h_1mr+(h_2+h_3)q^{\mu_1}sn}{q^{\mu_2}}\right),
\end{align*}
o\`u $a_0(\alpha,H) = \alpha$, $|a_h(\alpha,H)|\leq \min \left(\alpha,\frac{1}{\pi |h|}\right)$, selon le lemme de Vaaler (Lemme $1$ de~\cite{RudinShapiro}). Dans la somme $\sum_{1\leq s<q^{2\rho}}|S_4(r,s)|$, seule l'estimation des termes diagonaux ($h_0+h_1 = 0$) pour les petites valeurs de $|h_1|$ ($|h_1| \leq q^{2\rho} $) sera affect\'ee par nos changements. Pour les autres quantit\'es, l'estimation 
$\sum_{0\leq h<q^{\mu_2-\mu_0}}|\hat{g}(h)|^2=1$ suffit. Ainsi, comme dans \cite{RudinShapiro}, on se ram\`ene \`a une estimation de 
 \begin{equation}\label{S8}
S_8(r) = q^{2(\mu_2-\mu_0)}\ds\sum_{|h_1| \leq q^{2\rho}}|a_{h_1}(q^{\mu_0-\mu_2},H)|^2\min \left(q^{\mu},\ds\frac{q^{\mu_2}}{r|h_1|}\right)\ds\sum_{0 \leq h' < q^{\mu_2-\mu_0}}|\hat{g}(h'-h_1)\hat{g}(h')|^2.
\end{equation}

Le Lemme $11$ de \cite{RudinShapiro} permet ensuite \`a Mauduit et Rivat de majorer la somme des coefficients de Fourier en moyenne. Son analogue dans notre cas est l'estimation suivante :
\begin{lemma}\label{Nouveau lemme 10}
Soient $\mu$ et $\rho$ des entiers tels que $\mu \leq (2+4c/3)\rho $, o\`u $c$ est la constante introduite dans la d\'efinition \ref{propriete fourier mauduit rivat}. Alors, uniform\'ement pour $\lambda$ entier compris entre $(\mu_2-\mu_0)/3 $ et $4(\mu_2-\mu_0)/5 $ et $t$ r\'eel, on a  $$\ds\sum_{0 \leq k < q^{\mu_2-\mu_0-\lambda}}|\hat{g}(k+t)|^2 \ll (\gamma(\lambda)-\mu_1+\mu_0)  \, q^{(\mu_1-\mu_0-\gamma(\lambda))/2}(\log q^{\mu_2-\mu_1})^2 .$$
\end{lemma}

Ce lemme permet de dire que $$\ds\sum_{|h_1|\leq q^{2\rho}}\sum_{0 \leq h' < q^{\mu_2-\mu_0}}|\hat{g}(h'-h_1)\hat{g}(h')|^2 \ll (\gamma(\mu_2-\mu_0-2\rho)-\mu_1+\mu_0)  \, q^{-\frac{\gamma(\mu_2-\mu_0-2\rho)-\mu_1+\mu_0}{2}}  $$ ce qui conduit \`a  
$$\ds\frac{1}{q^{2\rho}}\ds\sum_{1 \leq r < q^{2\rho}}S_8(r) \ll (\gamma(\mu_2-\mu_0-2\rho)-\mu_1+\mu_0)  \, q^{\mu-\frac{\gamma(\mu_2-\mu_0-2\rho)-\mu_1+\mu_0}{2}}+q^{\mu-\rho}\log q^{\rho},  $$
 puis \`a
 \begin{align}\label{S4'}
 \ds\frac{1}{q^{3\rho}} & \ds\sum_{1 \leq r < q^\rho}\ds\sum_{1 \leq s < q^{2\rho}} |S_4(r,s)|\ll (\log q)^3(\mu+\nu)^3q^{\mu+\nu+3(\mu_2-\mu_0)+2\rho}(q^{-\nu}+q^{-\mu_2})
 \\\notag & + q^{\mu+\nu+\mu_1-\mu_0}\Big(\big(\gamma(\mu_2-\mu_0-2\rho)-\mu_1+\mu_0\big)  \, q^{-\frac{\gamma(\mu_2-\mu_0-2\rho)-\mu_1+\mu_0}{2}}+q^{-\rho}\log q^{\rho}\Big)
 \\ \notag & \quad \quad \quad (\tau(q^{\mu_2-\mu_1})+q^{\mu_2-\mu_1-\nu}\log q^{\mu_2-\mu_1}). 
 \end{align}

\bigskip

Attelons-nous \`a  d\'emontrer le Lemme \ref{Nouveau lemme 10}.

\begin{proof}[Preuve du Lemme \ref{Nouveau lemme 10} ]
On a par hypothèse $\mu_1-\mu_0 \leq \lambda \leq \mu_2 - \mu_0$, donc, en s\'eparant la somme d\'efinissant  $\hat{g}(t)$ selon les restes de la division euclidienne par $q^\lambda$, on obtient \begin{align*}
 \hat{g}(t)&
= \ds\frac{1}{q^{\mu_2-\mu_0-\lambda}}\ds\sum_{0 \leq v <q^{\mu_2-\mu_0-\lambda}}f(vq^{\mu_0+\lambda})e\left(-\ds\frac{vt}{q^{\mu_2-\mu_0-\lambda}}\right)
\\ & \quad \quad \times\ds\frac{1}{q^{\lambda}}\ds\sum_{0 \leq u < q^{\lambda}}f(uq^{\mu_0}+vq^{\mu_0+\lambda})\overline{f(vq^{\mu_0+\lambda}) f^{(\mu_1)}(uq^{\mu_0})}e\left(-\ds\frac{ut}{q^{\mu_2-\mu_0}}\right).
\end{align*}

Dans la ligne du bas, Mauduit et Rivat remplacent  $f$ par la fonction tronqu\'ee $f^{(\mu_0+\lambda+\rho_3)}$ associ\'ee, o\`u $\rho_3$ est un param\`etre qui sera optimis\'e. \`A nouveau, le traitement de la partie correspondant \`a la fonction tronqu\'ee est inchang\'e. Dans le traitement de l'erreur en revanche, on est amen\'e \`a majorer le cardinal de l'ensemble $\mathcal{W}_{\lambda}$ des entiers $w = u + vq^{\lambda}$ tels que $$f(uq^{\mu_0}+vq^{\mu_0+\lambda})\overline{f(vq^{\mu_0+\lambda})} \neq f^{(\mu_0+\lambda+\rho_3)}(uq^{\mu_0}+vq^{\mu_0+\lambda})\overline{f^{(\mu_0+\lambda+\rho_3)}(vq^{\mu_0+\lambda})}.$$ La faible propri\'et\'e de petite propagation (\ref{petite propagation}) donne alors $|\mathcal{W}_{\lambda}| \ll q^{\mu_2-\mu_0-\rho_3+\log(\rho_3)} $ et permet de d\'emontrer le lemme.
\end{proof}
\bigskip

Nous n'avons plus qu'\`a r\'eunir les \'equations \eqref{mes termes d'erreurs reunis}, \eqref{S3 a S4} et \eqref{S4'} et \`a utiliser $\mu_2 = \mu+2\rho$, $\mu_1 = \mu-2\rho$ et $\mu_0 = \mu_1-2\rho'$ pour obtenir :
\begin{align}
 \notag |S_{II}(\vartheta)|^4 &  \ll \max (\log (q^{\mu-2(\rho+\rho')}),\tau (q^{\mu-2(\rho+\rho')}))q^{4(\mu+\nu)-2\rho+2\log \rho}
\\ \label{grosse equa SII vert} & \quad + \max(\tau(q),\log q)(\mu+\nu)^{\omega(q)}q^{4(\mu+\nu)-2\rho'+\log (\mu-2\rho)}
\\ \label{grosse equa SII bleu} & \quad +  q^{3(\mu+\nu)}(\log q)^3(\mu+\nu)^3q^{\mu+\nu+3(2(\rho+\rho'))+2\rho}(q^{-\nu}+q^{-\mu_2})
\\ \label{grosse equa SII rose} & \quad + q^{3\mu+3\nu} \Big[ q^{\mu+\nu+2\rho'}(\gamma(2\rho+2\rho')-2\rho')  \, q^{-\frac{\gamma(2\rho+2\rho')-2\rho'}{2}}+q^{-\rho}\log q^{\rho})\Big.
\\ \notag & \quad \quad \quad \Big.(\tau(q^{4\rho})+q^{4\rho-\nu}\log q^{4\rho}) \Big] 
\end{align}

En utilisant $\gamma(x) \leq x/2$ (($26$) de~\cite{RudinShapiro}) et $\nu\geq 6\rho$, on obtient
\begin{equation*}
\eqref{grosse equa SII rose} \ll q^{4(\mu+\nu)+3/2(\mu_1-\mu_0)-\gamma(2\rho)}\log (q^{\rho})\tau(q)(\mu_2-\mu_1)^{\omega(q)}.
\end{equation*}

Notre estimation de $|S_{II}|^4$ devient de la m\^eme forme que celle trouv\'ee dans~\cite{RudinShapiro}, on peut alors d\'eduire, en faisant les m\^emes choix qu'eux :
\begin{equation}\label{controle de SII}
|S_{II}(\vartheta)|^4 \ll \max (\tau(q)\log q, (\log q)^3)(\mu+\nu)^{2+\log q+\max (\omega(q),2)}q^{4\mu+4\nu-\gamma(2\lfloor \mu / 15 \rfloor)}.
\end{equation}
 
\medskip

En rappelant $q^{\mu+\nu-1} \leq x < q^{\mu + \nu} $, en utilisant le Lemme $1$ de~\cite{MR:gelfond} avec les estimations (\ref{Estim SI nouvelle final}) et (\ref{controle de SII}), nous obtenons : 

\begin{equation*}
\left| \ds\sum_{x/q < n \leq x}\Lambda(n)f(n)e(\vartheta n) \right| \ll \left( \log x \right)^2 \left(S_I + S_{II} \right),
\end{equation*}

or par l'estimation (\ref{Estim SI nouvelle final}) 

\begin{equation*}
S_I(\vartheta)\ll (\log q)^{5/2}(\mu+\nu)^{2 + \log q}q^{\mu+\nu -\frac{\gamma((\mu+\nu)/3))}{2}}
\end{equation*}

et par l'estimation (\ref{controle de SII}) 

\begin{equation*}
|S_{II}(\vartheta)| \ll \max (\tau(q)\log q, (\log q)^3)^{1/4}(\mu+\nu)^{1/2+\log q/4+\max (\omega(q),2)/4}q^{\mu+\nu-\gamma(2\lfloor \mu / 15 \rfloor)/20},
\end{equation*}
 nous pouvons conclure la preuve de ce th\'eor\`eme comme le font Mauduit et Rivat.

\section{Applications}\label{Section : applis}

Dans cette partie nous appliquons les r\'esultats des parties~\ref{section petite propa} et \ref{section genealogie}  pour obtenir une large classe de fonctions qui v\'erifient un th\'eor\`eme des nombres premiers.

Nous avons vu que si une fonction v\'erifiait la faible propri\'et\'e de  petite propagation et la propriété de Fourier (\ref{equa a verifier pour MR}), alors elle v\'erifiait la majoration (\ref{nouvelle equation von mangoldt}). De plus, nous avons vu dans la partie~\ref{section petite propa} que les fonctions associ\'ees aux suites $\beta$-r\'ecursives v\'erifiaient la faible propriété de  petite propagation (Proposition \ref{Prop petite propa suite type RS}). Enfin, nous avons vu dans la partie~\ref{section genealogie}, que, pour une fonction associ\'ee \`a une suite $\beta$-r\'ecursive, v\'erifier la propriété de Fourier  (\ref{equa a verifier pour MR}) revenait \`a trouver $\alpha $ r\'eel et $\omega_1$, $\omega_2$ de taille $\beta - 1$ de m\^eme suffixes, et $k_1$ et $k_2$ de telle sorte que 
\begin{equation}\label{equa pour tout appliquer}
 \alpha \left(g(\omega_1\cdot k_1) - g(\omega_1 \cdot k_2) - g(\omega_2 \cdot k_1) + g(\omega_2\cdot k_2)\right) \notin \Z.
\end{equation}

Nous notons $K = K(g,\omega_1, \omega_2, k_1, k_2) = g(\omega_1\cdot k_1) - g(\omega_1 \cdot k_2) - g(\omega_2 \cdot k_1) + g(\omega_2\cdot k_2) $.

Nous allons ici donner des exemples de suites $\beta$-r\'ecursives, et montrer que pour certains $k_1, k_2, \omega_1, \omega_2 $ leurs fonctions de propagations v\'erifient (\ref{equa pour tout appliquer}) si et seulement si $\alpha$ n'est pas un entier. Il y a deux grandes classes de fonctions, que nous traitons s\'epar\'ement :

\subsection{Nombre d'occurences}

\begin{proposition}\label{proposition nombre d'occurence}
Soit $k \geq 2$, et soit $B \subset \M_k $  tel que  $ \exists \, \omega \in B$  de sorte que
 \begin{equation}\label{condition sur B a gauche}
\exists \, l_1  : l_1 \cdot \underline{\omega}_{(k-1)} \notin B,
\end{equation}

\begin{equation}\label{condition sur B a droite}
\exists \, l_2  :  \overline{\omega}^{(k-1)}\cdot l_2 \notin B,
\end{equation} 
et
\begin{equation}\label{condition sur B au milieu}
l_1\cdot \epsilon_{k-2}(\omega)\ldots\cdot\epsilon_1(\omega)\cdot l_2 \notin B.
\end{equation}
Soit $(a(n))_{n \in \N}$ une suite $k$-r\'ecursive de fonction de propagation $g = \ds\sum_{\chi \in B}\mathds{1}_{\chi}.$ Alors $K = 1$ et $(a_n)_{n \in \N}$ vérifie un théorème des nombres premiers.
\end{proposition}

\begin{proof}
En choisissant $\omega_1 = \overline{\omega}^{(k-1)}$, $ \omega_2 = l_1 \cdot \epsilon_{k-2}(\omega)\ldots \cdot \epsilon_1(\omega)$, $k_1 = \epsilon_0 (\omega)$, $k_2 = l_2 $ , on a $\omega_1 \cdot k_1 = \omega$ et donc :
\begin{equation*}
K = g(\omega) - g(\overline{\omega}^{(k-1)}\cdot l_2 ) - g(l_1 \cdot \underline{\omega}_{(k-1)}) + g(l_1\cdot \epsilon_{k-2}(\omega)\ldots\cdot\epsilon_1(\omega)\cdot l_2) = 1.
\end{equation*}

\end{proof}

De ce r\'esultat \textit{a priori}  tautologique, on tire de nombreuses cons\'equences. Le fait que $K$ soit \'egal \`a $1$ implique que (\ref{equa pour tout appliquer}) est \'equivalente \`a $\alpha$ non entier, ou encore que les suites $\beta$-r\'ecursives ayant une fonction de propagation correspondant aux conditions de la Proposition \ref{proposition nombre d'occurence} v\'erifient la propriété de Fourier (\ref{equa a verifier pour MR}) si et seulement si $\alpha$ n'est pas un entier.

Or ce type de fonction recouvre de nombreux cas classiques. Par exemple : 
\begin{enumerate}[label = (\Roman *)]
\item Si on prend $B = \{ \omega \}$, alors il est clair qu'il existe $\omega \in B $ v\'erifiant (\ref{condition sur B a gauche}), (\ref{condition sur B a droite}) et (\ref{condition sur B au milieu}), et si on pose $a(k) = 0  $ pour tout $k < q^{|\omega|} $, on trouve les suites qui comptent le nombre d'occurrences d'un mot quelconque de taille supérieure ou égale à $2$.
\item Soit $k \geq 1 $. Si on prend $B = \{a\cdot\gamma\cdot b , \gamma \in \M_k \} $, on trouve alors que $g = \mathds{1}_{a \cdot z \cdot b}$, et donc le nombre d'occurrences des mots de la forme $aZb$ o\`u $Z$ est un mot arbitraire. En particulier $q = 2, a = 1, b = 1$ donne la suite introduite par Allouche et Liardet dans~\cite{AlloucheLiardet}. On peut l'am\'eliorer de sorte \`a assigner des lettres fixes entre les deux extremit\'es en posant $B = \{a_0\cdot \gamma_0\cdot a_1 \cdots \gamma_k \cdot a_{k+1}, \,  \gamma_i \in \M_{\zeta (i)} \, \forall 0 \leq i \leq k \}$, où $\zeta$ est une fonction de $\N$ dans $\N$ arbitraire.
\item Si on suppose qu'on n'est pas dans le cas $q=\beta=2$, on peut prendre $B = \ds\bigcup_{a \in \M_1}\{a\cdot\ldots\cdot a\} $ pour compter le nombre d'occurence des mots de m\^eme taille et ayant une seule lettre, comme $000$, $111$ et $222$ pour $q = 3$ et $k = 3$ .
\end{enumerate}

Notre condition \eqref{equa pour tout appliquer} permet de traiter de nombreux cas classiques. En revanche la suite $(a(n))_{n \in \N}$ qui compte le nombre de mots $00$ et $11$ dans l'écriture de $n$ en base $2$ n'entre pas dans ce cadre. En effet, sous ces conditions, la fonction de propagation $g(a\cdot b)$ vaut $1$ si et seulement si $a = b= 0$ ou $a=b=1$ et alors quels que soient $\omega_1$, $\omega_2$ de taille $1$ et $k_1,k_2$ de taille $1$, on a \begin{align*}
|K| & = |g(\omega_1\cdot k_1) - g(\omega_2\cdot k_1) - g(\omega_1\cdot k_2) + g(\omega_2\cdot k_2)|
\\ &= |g(00)-g(01)-g(10)+g(11)|
\\ &= 2 \equiv 0 \bmod 2,
\end{align*} et $\alpha = 1/2$ est également une valeur proscrite.

\subsection{Polyn\^omes sur les chiffres}

Dans cette sous-partie, nous r\'esolvons partiellement la question pos\'ee par Kalai~\cite{Kalai2} à travers la démonstration du Théorème \ref{Thm Kalai}.

\begin{proof}[Demonstration du Th\'eor\`eme \ref{Thm Kalai}]
Soit $(x_{k-1},\ldots,x_1)\in\{0,1,\ldots,q-1\}^{k-1}$ tel que $$P_1(1,x_{k-1},\ldots,x_1,1)=1.$$ On pose $g(\omega)=P(\epsilon_{i+k}(n),\ldots,\epsilon_i(n))$ et $\omega_1 = 1\cdot x_{k-1} \cdot x_{k-2}\cdots x_1 $, $\omega_2 = 0\cdot x_{k-1}\cdot x_{k-2} \cdots x_1 $, $k_1 = 1 $ et $k_2 = 0 $. Alors 
\begin{align*}
  K&=g(1\cdot x_{k-1}\cdot x_{k-2}\cdots x_1 \cdot 1)-g(1\cdot x_{k-1}\cdot x_{k-2}\cdots x_1 \cdot 0)\\
	&\quad -g(0\cdot x_{k-1}\cdot x_{k-2}\cdots x_1 \cdot 1)+g(0\cdot x_{k-1}\cdot x_{k-2}\cdots x_1 \cdot 0)\\
	&=P_1(1,x_{k-1},\ldots,x_1,1)=1,
\end{align*}
et donc $\eqref{equa pour tout appliquer}\Leftrightarrow \alpha \notin \Z$. Par la Remarque \ref{Remarque gros thm}, le Théorème \ref{Thm Kalai} est démontré. 
\end{proof}

Le fait qu'on ne puisse pas traiter le cas $P(X_1,X_0) = X_1+X_0$ vient du fait que dans ce cas, quels que soient $x_0, x_0',x_1,x_1'$ \begin{align*}
|K| &= |P(x_1,x_0)-P(x_1',x_0)-P(x_1,x_0')+P(x_1',x_0')| 
\\ &= |x_1+x_0-x_1'-x_0-x_1-x_0'+x_1'+x_0'|
\\ &= 0
\end{align*} et cette remarque vaut pour tout polynôme $P(X_k,\ldots,X_0)$ bilinéaire en $X_k$ et $X_0$.

\end{document}